\documentclass[11pt]{article}
\usepackage{graphicx} 
\usepackage{float}
\usepackage[font=small,labelfont=bf,margin=1cm]{caption}
\usepackage{subcaption}
\usepackage{mathtools}
\usepackage{mathrsfs}
\usepackage{verbatim} 
\usepackage{amsmath,amssymb,amsthm}
\usepackage{bbm}
\usepackage{xcolor}
\usepackage{hyperref}
\usepackage[utf8]{inputenc}  
\usepackage[T1]{fontenc}
\usepackage{bm}
\hypersetup{
    colorlinks=true,
    linkcolor=blue,
    filecolor=blue,  
    citecolor=[RGB]{76,153,0},
    urlcolor=blue,
    }
\usepackage[english]{babel}
\usepackage{pgfplots}
\pgfplotsset{compat=1.18}
\usepgfplotslibrary{groupplots} 

\usepackage[top=1in,bottom=1in,left=1in,right=1in]{geometry}

\newcommand{\Bin}{\mathrm{Bin}}

\newcommand{\E}{\mathbb{E}}
\newcommand{\Prob}{\mathbb{P}}

\usepackage{etoolbox}
\makeatletter
\patchcmd{\@endtheorem}{\@endpefalse}{}{}{}
\makeatother

\newtheorem{theorem}{Theorem}
\numberwithin{theorem}{section}
\newtheorem{definition}[theorem]{Definition}

\newtheorem{corollary}[theorem]{Corollary}

\newtheorem{lemma}[theorem]{Lemma}

\title{Limit Laws for Consensus Protocols on the Complete Graph 
}

\author{Julian Becker\thanks{LMU Munich, Department of Mathematics, Theresienstr.~39, 80333 Munich, Germany. Email: \texttt{\{becker,kpanagio\}@math.lmu.de}.} \and Konstantinos Panagiotou\footnotemark[1]
}

\begin{document}

\maketitle

\begin{abstract}
We study a distributed consensus problem on a complete communication network of 
$n$ vertices, each holding one of two opinions. The vertices communicate in rounds, possibly in the presence of adversarial noise, and exchange information until they all agree on a single opinion. We consider a general class of protocols, where the vertices randomly sample neighbors and update their own opinion according to an update function $f$ depending on the sampled opinions. A prominent example is the \texttt{$k$-maj} protocol, where every vertex adopts the majority opinion of~$k$ randomly sampled neighbors, breaking ties uniformly.

We consider the \textit{runtime}~$R_n$ that is the number of rounds until all vertices agree on the same opinion, which we call the~\emph{dominating opinion}~$D_n$.
In our main result we describe the limiting distributions of these two key quantities for a large class of update functions $f$, for arbitrary initial configurations and under the presence of an adversary who may alter the opinions of up to $o(\sqrt{n})$ vertices in each round.
We show that there are $f$-specific constants $\gamma, m > 0$ such that~$R_n$ centers around $\mu_n = \frac{1}{2}\log_\gamma n + \log_m\ln n$, and we describe the asymptotic distribution of $R_n - \mu_n$. In particular, we show that it does not converge, and that it becomes asymptotically periodic both in the $\log n$ as well as the $\log\log n$ scale. Applied to
\texttt{$k$-maj}, our results show, among other things, that $\gamma_{\texttt{$k$-maj}} = \binom{k-1}{\lfloor k/2 \rfloor}2^{1-k}k \sim ({2k}/{\pi})^{1/2}$.
\end{abstract}

~

\section{Introduction}

In this paper we provide a tight analysis for the runtime of a large class of consensus protocols on fully connected networks that includes the well-studied~\texttt{$k$-maj} ($k$-majority) protocol. In particular, we describe the limiting distribution of the runtime and of the eventually dominating opinion in a general setting, where there is (not necessarily) a dominating opinion in the beginning, and where there might be adversarial changes to the opinion distribution.

The study of consensus protocols and related probabilistic dynamics -- where vertices in a network exchange information, synchronously or asynchronously, possibly in the presence of noise, in order to reach consensus -- has attracted extensive attention not only within distributed computing, but also among other disciplines, where opinion forming and synchronization are relevant primitives. The literature is vast, and we refer to the extensive overview by Becchetti et~al.~\cite{consensusprotoverview} for a broad discussion and many further references, and to the subsequent sections. 

\paragraph{Example: \texttt{$k$-maj}.} Before we proceed to our general model we give a brief demonstration of our results in the case of the \texttt{$k$-maj} protocol, which is defined on a complete communication network of $n$ vertices, each holding one of two opinions that we call $X$ and $Y$. The vertices communicate in rounds, where everyone samples independently $k$ neighbors uniformly at random and adopts the majority opinion among them, breaking ties uniformly. Let $R_n^{\texttt{$k$-maj}}(X_0)$ be the \emph{runtime}, i.e., the number of rounds until all vertices hold the same opinion, where $X_0$ is the initial number of vertices holding $X$. Moreover, let $D_n^{\texttt{$k$-maj}}(X_0)$ be the dominating opinion at the end of the process.

Majority protocols have been studied extensively, a small sample is~\cite{becchettilowerbound,Becchetti2016Stabilizing, hierarchyconj}. Berenbrink et al.~show in~\cite{inp:bcghkr23} that $R_n^{\texttt{$k$-maj}}(X_0)$ \emph{stochastically dominates} $R_n^{\texttt{$\ell$-maj}}(X_0)$ for $\ell > k$, so that the runtime should decline with increasing $k$.
Ghaffari and Lengler \cite{GhaffariLengler} established that with probability $1-n^{-O(1)}$, \texttt{$3$-maj} finishes in $O(\ln n)$ rounds, which then implies that \texttt{$k$-maj} reaches consensus in logarithmic time for any $k\geq 3$.
Yet, the magnitude of the speedup and the precise effect of $k$ are not studied. 
In the following, and throughout the paper, we use the abbreviation \emph{whp} (with high probability) to denote events that occur with probability $1-o(1)$. Different notions of whp appear in the literature,~e.g.,~with probability $1-n^{-O(1)}$, but we do not distinguish between them here. 

Let $k\geq 3$ and $X_0 = n/2 + d \sqrt{n}$, where $0 \le d \le (1/2-\varepsilon)\sqrt{n}$ for some $0<\varepsilon \leq 1/2$. This assumption is not restrictive. Indeed, if $X$ is initially the minority opinion, we can just swap $X$ and $Y$ to guarantee that $d \ge 0$. Moreover, if $d = (1+o(1))\sqrt{n}$, then already at the beginning a vanishing fraction of the vertices has opinion $Y$, which is not so interesting. Let
\[
m = \lceil k/2\rceil, \quad \beta =\binom{k}{\lfloor k/2 \rfloor} \quad \text{and} \quad \gamma = \binom{k-1}{\lfloor k/2\rfloor} 2^{1-k}k.
\]
We say that a sequence of random variables $(W_n)_{n\in \mathbb{N}_0}$ is \emph{stochastically bounded} by $(a_n)_{n \in \mathbb{N}_0}$, and we write~$W_n = O_\Prob(a_n)$, if for every $\varepsilon>0$ there exist $C_\varepsilon, n_\varepsilon$ such that $\Prob(|W_n|\ge C_\varepsilon a_n) \le \varepsilon$ for every $n\geq n_\varepsilon$. 
Set
\[
\mu_n^{\text{$k$-maj}}(d) :=
\begin{cases}
    \frac{1}{2}\log_\gamma n + \log_m \ln n, & \text{if } d = O(1), \\[1mm]
    \frac{1}{2}\log_\gamma(n/d^2) + \log_m \ln n, & \text{if } d = \omega(1).
\end{cases}
\]
Our main result, Theorem~\ref{thm: main result}, implies that 
\[
    R_n^{\texttt{$k$-maj}}(X_0) = \mu_n^{\texttt{$k$-maj}}(d) + O_\Prob(1),  
\]
that is, the runtime has \emph{stochastically bounded} fluctuations around $\mu_n^{\texttt{$k$-maj}}(d)$.
Actually, our results imply more. If $d = \omega(1)$, 
in which case there is already a substantial initial bias towards $X$,
then we can show that the runtime concentrates on at most \emph{two} values.
Moreover, whp  $D_n^{\texttt{$k$-maj}}(X_0) = X$ in that case. On the other hand, if $d = O(1)$, the picture (and the analysis) is more intricate, as the process has initially no/very small drift. Here we describe and provide qualitative information about the limiting distribution of $R_n^{\texttt{$k$-maj}}(X_0) - \lfloor\mu_n^{\texttt{$k$-maj}}(d)\rfloor$, which has support on all~$\mathbb{Z}$, see Section~\ref{ssec:model}. Moreover, in this case the two opinions perform a race, and both have a chance to prevail; to wit, we obtain  completely explicit that
\[
    \Prob\big(D_n^{\texttt{$k$-maj}}(X_0) = X\big) = \sqrt{\frac{2(\gamma^2-1)}{\pi}}\int_0^\infty e^{-2(\gamma^2-1)(t-d)^2} dt + o(1) \in (0,1).
\]
Thus, our results enable us to perform a fine-grained runtime analysis and a qualitative comparison of majority protocols for all values of $k$. Let us also mention that our analysis is robust enough and leaves enough margin that it allows for small changes in the model, namely adversarial interventions of order $o(\sqrt{n})$ in each round. We make this more precise in the following section.

\subsection{Model \& Results}
\label{ssec:model}

Let $f:[0,1] \rightarrow [0,1]$ be, for the moment, an arbitrary function. We associate to $f$ a consensus protocol that is a stochastic process $(X_t)_{t\in\mathbb{N}_0}$ with initial configuration
\begin{equation*}
    X_0 = \frac{n}{2} + d \sqrt{n} +o(\sqrt{n}).    
\end{equation*}
Thus, at $t=0$, there are $X_0$ vertices holding opinion $X$, and there is an initial bias of  $d\sqrt{n}$ (vertices). As we will see, the most interesting and complex case is when $d$ is bounded and independent of $n$, since then the process moves just a little or even not at all in expectation; if, on the other hand, the initial bias is larger, then our analysis applies as well but starts already in a 'simpler' situation. Now $f$ comes into play. Given $X_t$, the probability that a specific vertex has opinion $X$ at time $t+1$ is given by $f(X_t/n)$, and this is independent of all other vertices. In other words, the probability that a vertex maintains or changes its opinion depends only on the fraction of vertices that have adopted $X$ through $f$; this resembles the fact that we consider protocols where vertices sample the opinions of their neighbors. Formally, the process evolves according to the recursion
\begin{equation*}
    X_{t+1} \stackrel{d}{=} \Bin\big(n,f(X_t/n)\big), \quad t \in \mathbb{N}_0.
\end{equation*}
We also write $Y_t = n - X_t$ for the number of vertices that have adopted opinion $Y$ at time $t$. We denote by $R_n(X_0)$ the runtime and by $D_n(X_0)$ the eventually dominating opinion, so that
\begin{equation*}
    R_n(X_0) := \inf\big\{t\in \mathbb{N}_0 : X_t =n \text{ or } Y_t =n\big\}
    \quad\text{and}\quad
    D_n(X_0):=\begin{cases}
        X,\quad &\text{if } X_t=n \text{ for some } t,\\
        Y,\quad &\text{if } Y_t=n \text{ for some } t.
    \end{cases}
\end{equation*}
Note that in all preceding definitions it is not required that $X_0 \in \mathbb{N}_0$ -- any real number in $[0,n]$ works. This will be sometimes helpful in later formulations. Throughout, we consider only specific~$f$.
\begin{definition}\label{def: majority-type}
A function $f:[0,1] \rightarrow [0,1]$ is a \emph{majority-type update function}, if
{
\begin{itemize}
\itemsep-1pt
    \item $f$ is $(m+1)$-times continuously differentiable for some $m\in \mathbb{N}_{\geq 2}$; 
    \item $f$ is convex on $[0,1/2]$;
    \item $f$ is symmetric around $1/2$, i.e., $1-f(1-x)=f(x)$ for $x\in[0,1]$;
    \item there exist $\gamma >1$ and $\beta >0$ such that, as $x \to 0$,
    \[
    f(x) = \beta x^m +O(x^{m+1})\quad \text{and}\quad f(1/2 +x) = 1/2 +\gamma  x+O(x^2).
    \]
\end{itemize} 
}
\end{definition}
For $i\in \mathbb{N}_0$, we denote the $i$-th iterate of a majority-type update function by
\[
f^{(0)}(x)=x
\quad 
\text{and}
\quad
f^{(i)}(x)=f\circ f^{(i-1)}(x), \quad  \text{for }i\geq 1.
\]
The technical assumptions on $f$ reflect the nature of the processes that we want to study and that are \emph{reasonable} to study. Indeed, since $f(0)=0$ and $f(1)=1$ the consensus states are absorbing, which emphasizes the term \emph{runtime}. By symmetry of $f$, starting with $X_0=n/2$, it holds that $\Prob(X_t =0)=\Prob(X_t=n)$ for every $t\geq0$; in particular, the process has no inherent drift towards one of the two opinions (the biased case is much easier to study by basic methods). This assumption will later allow us to show that $X_t \approx n/2 + N_t$, where $N_t$ follows a normal distribution with increasing variance in $t$. We will see that after some point in time $t$, when the process has accumulated enough bias, $X_t$ is concentrated around its mean, and so the long-term behavior will be described by iterates of $f$. 
Therefore, the assumption that $f$ is convex, behaves polynomially around $0$ and linearly in $1/2$ ensures that we understand iterates of $f$ well enough. 

The aforementioned \texttt{$k$-maj} protocols all give rise to majority-type update functions for specific~$m,\beta,\gamma$, see also Figure~\ref{fig:f_k} and at the beginning of the introduction for the actual values of the parameters. Moreover, there are also other protocols that fall in the same class. We have collected some examples in Section~\ref{sec:examples}, together with the computations leading to the values of $m,\beta,\gamma$ in each case:
\texttt{$k$-maj}, \texttt{rand-$k$-maj}, where each vertex samples a random number of neighbors and adopts the majority, and \texttt{$k$-neighb-rand}, where every vertex samples $k$ neighbors, and adopts $X$ if a specific random fraction of them has opinion $X$.

\begin{figure}[tb]
    \centering
\begin{tikzpicture}
\begin{groupplot}[
    group style={
        group size=2 by 1,
        horizontal sep=2cm,
    },
    width=7cm,
    height=5.5cm,
    tick label style={font=\small},
    label style={font=\small},
    legend style={font=\small},
    grid=both,
    grid style={dashed, gray!50, line width=0.2pt},
]

\nextgroupplot[
    xlabel={percentage of opinion $X$},
    ylabel={prob of adopting opinion $X$},
    legend pos=north west,
]
\addplot [color=blue, very thick] table [x=x, y=f3] {func_file.dat};
\addlegendentry{$f_3$}

\addplot [color=red, very thick] table [x=x, y=f5] {func_file.dat};
\addlegendentry{$f_5$}

\addplot [color=green, very thick] table [x=x, y=f7] {func_file.dat};
\addlegendentry{$f_7$}


\addplot [color=orange, very thick] table [x=x, y=f101] {func_file.dat};
\addlegendentry{$f_{101}$}

\nextgroupplot[
    xlabel={$x$},
    ylabel={density of $W$},
    legend pos=north east,
    xmin=-2, xmax=12,
]

\addplot [color=green, very thick] table [x=x, y=g7] {hxy.dat};
\addlegendentry{$k=7$}

\addplot [color=red, very thick] table [x=x, y=g5] {hxy.dat};
\addlegendentry{$k=5$}

\addplot [color=blue, very thick] table [x=x, y=g3] {hxy.dat};
\addlegendentry{$k=3$}


\end{groupplot}
\end{tikzpicture}
\caption{The left plot depicts the function $f_k$ for different $k$, which is the probability that a vertex adopts opinion $X$ in the case of \texttt{$k$-maj}. The right plot shows the probability density function of $W=h(-\log_\gamma|Z|)$ occurring in Corollary~\ref{cor: main 1} and Lemma~\ref{lem: tails Z} for several $k$.}
\label{fig:f_k}
\end{figure}
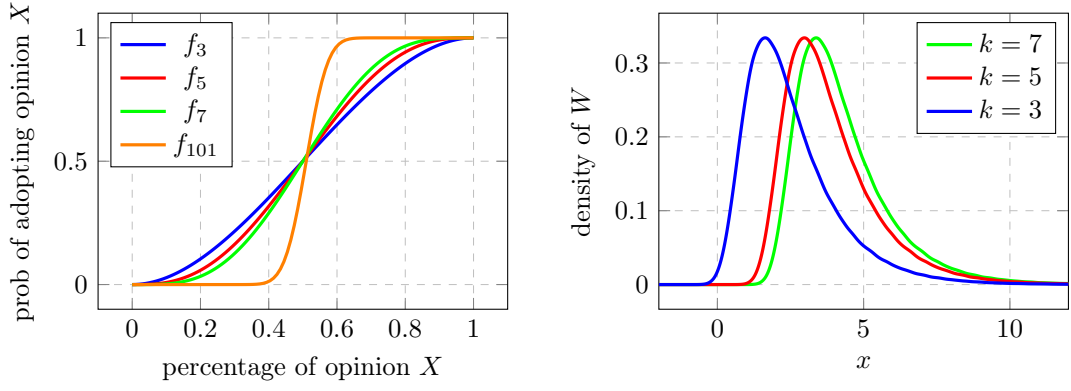

Our first main result pins down the probability that eventually all vertices agree upon some specific opinion. Here and in the sequel we denote by ${\cal N}(\mu, \sigma^2)$ a gaussian random variable with mean $\mu$ and variance $\sigma^2$.
\begin{theorem}\label{thm: X wins}
Let $d\in\mathbb{R}$, $f$ be a majority-type update function and  $X_0=n/2 + d\sqrt{n}+o(\sqrt{n})$ be the initial number of vertices with opinion $X$. Let $Z \stackrel{d}{=} \mathcal{N}\big(d,1/4(\gamma^2-1)\big)$. Then, as $n\to \infty$,
\[
    \Prob\big(D_n(X_0) = X\big) = \Prob\big(Z \geq 0\big)+o(1)
    \quad
    \text{and}
    \quad
    \Prob\big(D_n(X_0) = Y\big) = \Prob\big(Z \leq 0\big)+o(1).
\]
\end{theorem}
Thus, both opinions have a chance of becoming eventually dominating, even if one of the opinions has a (slight) dominance in the beginning. Our second main result describes the limit distribution of the runtime. As preparation, for a majority-type update function $f$ we define the auxiliary function
\[
    g:\mathbb{R}\to \mathbb{R},
    \quad 
    x \mapsto 1-x+\lim_{a\to \infty, a\in\mathbb{N}} \lim_{b \to \infty, b\in \mathbb{N}} b-a- \log_m|\ln f^{(b)}(1/2-\gamma^{-a-x})|.
\]
In Section \ref{subsec: technical preliminaries}
we will see that $g$ is well-defined and continuous. Moreover, we prove that $g$ is $1$-periodic and $\sup g-\inf g < 1$. 
With this at hand we obtain the following result.

\begin{theorem}
\label{thm: main result}
Let $d\in\mathbb{R}$, $f$ be a majority-type update function and  $X_0=n/2 + d\sqrt{n}+o(\sqrt{n})$ be the initial number of vertices with opinion $X$. Let $Z \stackrel{d}{=} \mathcal{N}\big(d,1/4(\gamma^2-1)\big)$. Then, as $n\to \infty$,
\[
    \sup_{s\in\mathbb{Z}}
    \left| \Prob\left(R_n(X_0) \geq s\right) -
    \Prob\Big(
        Z_n + \log_m\ln n + g(Z_n) \geq s
    \Big)
    \right|
    = o(1), 
    \quad
    \text{where}
    \quad
    Z_n :=\tfrac{1}{2}\log_\gamma(n/Z^2).
\]
\end{theorem}
The theorem has various immediate consequences. Without digesting the full details, it readily 
establishes that the runtime centers around $\frac12\log_\gamma n + \log_m\ln n$ and that it has stochastically bounded fluctuations (that we can describe precisely in terms of a (function of a) gaussian random variable).
\begin{corollary}\label{cor: main 2}
In the situation of Theorem~\ref{thm: main result},
\[
    R_n(X_0)=  \tfrac{1}{2}\log_\gamma n + \log_m\ln n + O_\Prob(1).
\]
\end{corollary}
To get a very informal justification of why such a result is expected, note that $f$ being a majority-type update function implies for small $x$ that a) if the fraction of vertices with opinion $X$ is $1/2 + x$, then after one round it is in expectation roughly $1/2 + \gamma x$ and b) if it is $1-x$, then after one round it is roughly $1-\beta x^m$. So, if we start with $x = \Theta(1/\sqrt{n})$, then after about $\frac{1}{2}\log_\gamma n$ rounds the fraction will be significantly larger than $1/2$, and after a constant number of rounds it will be close to 1. Then, due to b), a doubly exponential race to 1 will be initiated, resulting in the doubly logarithmic term. The fine details and the fluctuations of the process are then captured by Theorem~\ref{thm: main result}.

In a very reduced form, Corollary~\ref{cor: main 2} implies that, whp, $R_n(X_0) \sim \frac12\log_\gamma n$.
In particular, this yields the coarser bound $R_n(X_0) =\Theta(\ln n)$ whp, which was the best known result so far (for \texttt{$k$-maj}).
Apart from that, and this is not immediately apparent, a further and direct consequence of our main result is that there is \emph{no} limiting distribution, as the fractional parts of $\frac{1}{2}\log_\gamma n$ and $\log_m\ln n$ vary with $n$;
in order to obtain convergence, we have to restrict to specific subsequences of $n$.
We obtain the following corollary.
\begin{corollary}\label{cor: main 1}
Let $x,y\in[0,1)$ and $(n_i)_{i\in\mathbb{N}}$ be an increasing sequence in $\mathbb{N}$ such that, as $i\to\infty$,
\[
    \tfrac{1}{2}\log_\gamma n_i -\left\lfloor\tfrac{1}{2}\log_\gamma n_i \right\rfloor \to x
    \quad\text{and}\quad
    \log_m\ln n_i - \lfloor\log_m\ln n_i  \rfloor \to y.
\]
Then, in the setting of Theorem~\ref{thm: main result} and with $h(x)=g(x)+x$, in distribution, 
\[
    R_{n_i}(X_0) - \left(\left\lfloor \tfrac{1}{2}\log_\gamma n_i\right\rfloor + \lfloor \log_m\ln n_i\rfloor\right) \to \lfloor h(x-\log_\gamma|Z|) + y\rfloor.
\]
\end{corollary}
We can extract some information about the limiting distribution, see also Figure~\ref{fig:f_k}. In particular, it has the following (asymmetric) tail properties.
\begin{lemma}\label{lem: tails Z}
Let $x,y\in [0,1)$ and $H\stackrel{d}{=}\lfloor h(x-\log_\gamma|Z|) + y\rfloor$. Then, as $t\to \infty$,
\[
    \Prob(H\geq t) = \Theta\left(\gamma^{-t}\right)
    \quad\text{and}\quad 
    \Prob(H\leq -t) = \exp\left(\Theta\left(-\gamma^{2t}\right)\right).
\]
\end{lemma}
Up to now we have only considered initial configurations where the bias is $O(\sqrt{n})$. However, our forthcoming analysis also applies to the case where we start with a bias of $\omega(\sqrt{n})$ (and this is actually the less elaborate case). Here we obtain the very strong result that the runtime concentrates on (at most) two values.
\begin{corollary}\label{cor: concentration for large d} 
Let $f$ be a majority-type update function and  $X_0=n/2 + d\sqrt{n}$ be the initial number of vertices with opinion $X$, where $d = \omega(1)$ and $0<d \le (1/2-\varepsilon)\sqrt{n}$ for some $0<\varepsilon < 1/2$.
Let $$s_{n,d}:=\left\lceil \tfrac{1}{2}\log_\gamma (n/d^2)+\log_m\ln n+g\left(\tfrac{1}{2}\log_\gamma(n/d^2)\right)\right\rceil.$$
Then, whp,
\[
    D_n(X_0) = X
    \quad\text{and}\quad
    R_n(X_0)\in \{s_{n,d}-1, s_{n,d}\}.
\]
\end{corollary}
Finally, our analysis is robust enough to deal with an adversary who may change up to $o(\sqrt{n})$ opinions in every round; we elaborate on this more in the proof overview.

\paragraph{Further Related Work.} Apart from the previously described \texttt{$k$-maj} protocols for $k\geq 3$, the \texttt{$1$-maj} protocol, also known as the \texttt{voter model}~\cite{cooperetal,HASSIN2001248,kanade,nakata}, and the \texttt{$2$-choice} protocol \cite{cooper_et_al:LIPIcs.DISC.2017.13,crucianietal,cruciani:hal-02002462,twochoicestabconsensus}, where every vertex adopts the majority of its own opinion and that of two randomly sampled neighbors, have been extensively studied.
The runtime decreases from being linear in case of the \texttt{voter model} \cite{berenbrink2016boundsvotermodeldynamic} to logarithmic in the \texttt{$2$-choice} setting \cite{GhaffariLengler}. 
In the aforementioned work \cite{GhaffariLengler}, improving upon previous work \cite{becchettilowerbound,hierarchyconj,ElsasserFKMT16,elsaesseretal}, Ghaffari and Lengler show that even in the presence of an adversary of size $O(\sqrt{n})$, both \texttt{$2$-choice} and \texttt{$3$-maj} still finish in $O(\ln n)$ rounds whp. Their results in fact apply to a more general setting in which the number of opinions $\lambda=\lambda(n)$ exceeds two. In particular,
for \texttt{$2$-choice} with $\lambda=O(\sqrt{n/\ln n})$ and \texttt{$3$-maj} with $\lambda =O(n^{1/3}/\sqrt{\ln n})$, they show that the protocols finish whp in $O(\lambda\ln n)$ rounds even in the presence of an $O(\sqrt{n}/\lambda^{1.5})$ adversary. 
Shimizu and Shiraga \cite{shimizu25} prove analogous results without an adversary for a larger range of $\lambda$. Moreover, for arbitrary $\lambda$, they show that \texttt{$2$-choice} reaches consensus whp in $O(n\ln^3 n)$ rounds, while \texttt{$3$-maj} does so in $O(\sqrt{n}\ln^2n)$ rounds. They further establish that, for \texttt{$2$-choice} with $\lambda=o(n/\ln^2n)$ and $\alpha_1=O(1)$ depending on $\lambda$, any bias of size $\omega(\sqrt{\alpha_1 n\ln n})$, and for \texttt{$3$-maj} with $\lambda=o(\sqrt{n}/\ln n)$, any bias of size $\omega(\sqrt{n\ln n})$,
determines whp the dominating opinion.
Assuming an initial bias, d'Amore~et~al.~\cite{damore_et_al:LIPIcs.DISC.2025.27} show that \texttt{$k(n)$-maj} with $\lambda=o(n/\ln n)$ opinions reaches consensus whp in $O(\ln n)$ rounds, provided $k(n)=\omega(\lambda\ln n)$.

General frameworks for whole \emph{classes} of consensus protocols have also been studied. 
Schonebeck and Yu \cite{consesusIPS} consider ``majority-like'' protocols, where in each round a chosen vertex randomly updates its opinion. The probability of adapting an opinion is given by a \textit{update function} $f$ applied to the fraction of vertices with that opinion. The function $f$ has three main properties, namely monotonicity, symmetry and it has fixed points in $0$ and $1$. These properties are essentially satisfied by all \emph{reasonable} consensus protocols. The main result in~\cite{consesusIPS} states that majority-like protocols finish whp in $\Theta (n\ln n)$ rounds in dense Erd\H{o}s-R\'enyi graphs.
Moreover, Shimizu and Shiraga~\cite{expander} study ``quasi-majority'' protocols on expander graphs. Their requirements on the update function are nearly identical to the majority-like case.
Notably, (quasi-)majority-like protocols include \texttt{$k$-maj} as a special case and both classifications of update functions are close to our definition.

Variants of consensus protocols that impose additional structure on the set of opinions have also been studied. Prominent examples are the \texttt{stabilizing} process \cite{angluinetal} and the \texttt{$3$-median} protocol \cite{twochoicestabconsensus}, where every vertex adopts the median opinion of three random neighbors. These protocols require a total order on the $\lambda$ possible opinions. The \texttt{$3$-median} protocol reaches consensus whp in $O(\ln \lambda \ln\ln n + \ln n)$ rounds. In averaging dynamics, vertices may generate new opinions by adopting the (possibly weighted) average of their current opinion and those of randomly sampled neighbors \cite{boydetal,shahbook}.  Another popular direction investigates the impact of varying the underlying network. Recent examples include studies that examine different graph topologies, such as expander graphs \cite{expander, expanderserdosrenyi}, random graphs like stochastic block models and Erd\H{o}s-R\'enyi graphs \cite{binomialrandomgraph,consesusIPS,stochasticblockmodel,expanderserdosrenyi}, grid and torus structures \cite{gridmaj,torus}, dense graphs \cite{densegraphs}, regular graphs \cite{regulargraphtwochoice} and core–periphery networks \cite{crucianietal}. These represent only a subset of the many directions explored in the literature.

\paragraph{Open problems.} 
There are several natural research directions that can be pursued: allowing more than two opinions, allowing a growing number (in $n$) of neighbors that can be sampled, and investigating more powerful adversaries. Some of the techniques developed here will carry over to these settings, for example the study of the fine behavior and the fluctuations after a bounded number of rounds, or the fact that the process follows essentially a deterministic trajectory when one opinion is dominant. However, studying more opinions or larger neighborhoods will certainly introduce significant technical challenges, especially regarding the analysis of iterated multi-dimensional functions. 
Another research question that is left open here concerns the convergence of higher moments of the runtime. While our results describe the asymptotic distribution, it would be interesting to study stronger notions of convergence. One possible direction is to investigate large deviation bounds for the runtime around its expectation, as in~\cite{doerr_et_al:LIPIcs.ICALP.2017.138} for rumor spreading protocols.

\subsection{Proof Overview}\label{sec: proof overview}

The analysis is divided into three main parts. First, we establish a central limit theorem (CLT), which states that after any bounded number of rounds, the process, appropriately centered and normalized, converges to a normal distribution. Second, we show that afterwards the process stays sufficiently close to iterates of the function $f$ until it reaches almost unanimity.
Lastly, we show that we require $\log_m\ln n$ more time steps before reaching a consensus state. In what follows we present a detailed outline of all steps within  the proof of Theorem~\ref{thm: main result}, and we assume that we are in the situation of that theorem.

\medskip
\noindent
In our initial configuration the two opinions are (almost) evenly distributed between all vertices. Our first milestone is the following CLT that describes the distribution of $X_t$ after a bounded number of rounds. The proof is given in Section~\ref{subsec: Proof of CLT}.
\begin{lemma}\label{lem: CLT}
    For any $t \in \mathbb{N}_0$, as $n \to \infty$,
    \[
    \frac{X_t-n/2}{\sqrt{n}} \stackrel{d}{\to} \mathcal{N}\left(\mu_t, \sigma_t^2 \right), 
    \quad 
    \text{where} 
    \quad 
    \mu_t = \gamma^t d
    \quad
    \text{and}
    \quad
    \sigma_t^2 = \frac{\gamma^{2t}-1}{4\left(\gamma^2-1\right)}.
    \]
\end{lemma}
\noindent
We observe that while the mean of $X_t$ stays close to $n/2$ -- in particular, if $d=0$, then $\mu_t = 0$ always -- the variance takes off. So, although the process may not have a significant \emph{drift}, the variance becomes eventually so big, that a large deviation from $n/2$ becomes likely. In order to formalize this observation,
let $t \in \mathbb{N}_0$ and define $M_t := \max\left \{X_t,Y_t\right\}$ as the number of vertices with majority opinion at time $t$. Since $Y_t = n- X_t$, we have $M_t = n/2 +|n/2-X_t|$. Thus, the previous lemma not only characterizes the distribution of $X_t$, but also of $Y_t$ and $M_t$. In fact, by applying the continuous mapping theorem, we obtain 
$(M_t-n/2)/\sqrt{n}\stackrel{d}{\to} |\mathcal{N}\left(\mu_t, \sigma_t^2 \right)|$. Let $C>1$ and define
\begin{equation}
\label{eq:t1}
    t_1 := \lfloor \log_\gamma C\rfloor
    \quad
    \text{and}
    \quad
    Z_{C} \stackrel{d}{=} \mathcal{N}(\mu_{t_1}/C,\sigma_{t_1}^2/C^2).
\end{equation}
We will choose $C$ (very) large later on, so that specific properties are satisfied.
By basic properties of the normal distribution, Lemma \ref{lem: CLT} guarantees that
\begin{equation}
\label{eq:cltM}    
    \sup_{x\in\mathbb{R}}\big|\Prob(M_{t_1}\geq x\big) - \Prob\big(n/2+|Z_C|C\sqrt{n}\geq x\big)\big| = o(1),
\end{equation}
which says that at $t_1$, we expect a bias of order $C\sqrt{n}$ towards the one or the other opinion. 
Allowing an adversary who may change the opinion of up to $o(\sqrt{n})$ vertices every round obviously does not affect this statement, as Lemma \ref{lem: CLT} is robust with respect to changing $o(\sqrt{n})$ opinions.
Note that, as $C \to \infty$ with $\log_\gamma C\in \mathbb{N}$, we get $Z_C \stackrel{d}{\to} Z$, where $Z\stackrel{d}{=}\mathcal{N}(d,1/4(\gamma^2-1))$.
From now on, assume without loss of generality that opinion $X$ is the majority opinion at time $t_1$, i.e.,
\begin{equation*}
M_{t_1} = X_{t_1}\quad \text{and define}\quad \tilde X_{t_1} := X_{t_1}-n/2.
\end{equation*}
As already described, the main motivation for introducing $t_1$ and letting $C$ get large is that at time~$t_1$ we will observe a non-negligible bias, so that from that point on, the process will follow essentially a deterministic trajectory. Our second milestone establishes this by showing that if we start at $t_1$, the number of vertices with opinion $X$ in subsequent time steps can be described quite precisely by iterates of~$f$. 
The proof of the next lemma is in Section~\ref{subsec: Proof of Middle Phases}.
\begin{lemma}\label{lem: MiddlePhases}
    Let $c_\gamma= 1 + \gamma/(\gamma-1)^2$. Then, for any $C> 1$ the event ${\cal E}_C$ occurs with probability at least $1-e^{-C}$, where
    \[
        {\cal E}_C = 
        \bigcap_{t\geq 0} \left \{ \left| X_{t_1+t} - f^{(t)}\left(X_{t_1}/n\right)n\right | \leq c_\gamma \gamma^t \sqrt{Cn} \right \}.
    \]
\end{lemma}
The statement of the lemma also holds in the presence of an adversary that can modify in each round the opinions of $o(\sqrt{n})$ vertices, possibly by replacing $c_\gamma$ with a slightly larger constant. 
The analysis of distributed algorithms via iterates of appropriate functions appears also at other places, see~\cite{Push,Pull} for similar settings. Although it is mostly not possible to determine the iterated function, it is usually sufficient for the further analysis to know the behavior near its fixed points.
Indeed, by the properties of $f$, we can deduce for small $x$ and $t \in \mathbb{N}$ that $f^{(t)}(1/2+x)\approx 1/2+\gamma^tx$. Hence, if we write $\{z\} = z - \lfloor z \rfloor$ for the fractional part of $z$, informally
\[
f^{(s)}(X_{t_1}/n)
=
f^{(s)}(1/2 + \tilde X_{t_1}/n)
\approx 1/2 +\gamma^{-a-\{\log_\gamma(n/\tilde X_{t_1})\}},
\;
\text{where}
\;
s=\lfloor\log_\gamma(n/\tilde X_{t_1})\rfloor-a,
~a \in \mathbb{N}.
\]
That is, at time $t_1 + \lfloor\log_\gamma(n/\tilde X_{t_1})\rfloor-a$ the bias should be \emph{linear in $n$} and of magnitude close to $\gamma^{-a}n$. The following lemma makes this precise, its proof is in Section~\ref{subsec: Proof of T_2}.
\begin{lemma}\label{lem: T_2} 
There exist constants $c >0$ and $a_0\in\mathbb{N}$ such that for all $a\geq a_0$ the following holds. Consider the random time and fractional part
\begin{equation}
\label{eq:t2}    
    T_2  := t_1 + \lfloor\log_\gamma(n/\tilde X_{t_1})\rfloor-a
    \quad
    \text{and}
    \quad
    \Phi :=  \{\log_\gamma(n/\tilde X_{t_1})\}.
\end{equation}
Then, whp, $t_1 \le T_2 < \infty$ 
and
\begin{equation}\label{eq:t2-2}
    \frac12 + \gamma^{-a-\Phi} (1-c\gamma^{-a}) \leq f^{(T_2-t_1)}(X_{t_1}/n) \leq \frac12 + \gamma^{-a-\Phi}.
\end{equation}
\end{lemma}
After $T_2$ we perform $b \in \mathbb{N}$ more steps.
The function $f$ is monotonically increasing, since it is of majority-type and therefore convex on $[0,1/2]$ and concave on $[1/2,1]$.
Thus, by Lemma~\ref{lem: T_2}, for sufficiently large $a\in\mathbb{N}$, whp,
\[
    f^{(b)}\big(1/2+\gamma^{-a-\Phi}(1-c\gamma^{-a})\big)
    \leq f^{(T_2-t_1+b)}\big(X_{t_1}/n\big)
    \leq f^{(b)}\big(1/2+\gamma^{-a-\Phi}\big).
\]
Let
$T_3 := T_2 + b$. 
By applying Lemma~\ref{lem: MiddlePhases} and using that $f(1/2+x)\geq 1/2$ for $x\in[0,1/2]$, we obtain that there exists a (possibly different) $c>0$ such that for $a\in \mathbb{N}$ large enough and $b\in \mathbb{N}$, with probability tending to one as $n,C\to \infty$,
\[
    \tilde L n\leq X_{T_3} \leq \tilde Un,
\]
where
\begin{equation*}
\begin{aligned}
    \tilde L &= f^{(b)}\big(1/2+\gamma^{-a-\Phi}(1-c\gamma^{-a})\big) \cdot (1-c\gamma^{-a+b}\sqrt{Cn}/\tilde X_{t_1}), \\
    \tilde U &= f^{(b)}\big(1/2+\gamma^{-a-\Phi}\big)\cdot(1+c\gamma^{-a+b}\sqrt{Cn}/\tilde X_{t_1}).
\end{aligned}
\end{equation*}
This looks messy, but note that by the CLT property in~\eqref{eq:cltM} we know that $\sqrt{Cn}/\tilde X_{t_1}$ tends in distribution to $1/\sqrt{C}|Z_C|$, which, in turn, by~\eqref{eq:t1}, tends in probability to $0$ as $C\to \infty$.
We use all this to prove the following streamlined statement in Section~\ref{subsec: proof of T_2 consequences}.
In what follows, and throughout the paper, we write $o^\star(1)$ for a term tending to zero as $a\to\infty$, provided that $b,C,n$ are chosen sufficiently large.
More precisely, for $F:\mathbb{N}\times \mathbb{N}\times \mathbb{R}_{>1}\times \mathbb{N}\to \mathbb{R}$, 
\[  
    F(a,b,C,n) = o^\star(1)
    \quad\Longleftrightarrow\quad
    \limsup_{a\to\infty} \limsup_{b\to\infty} \limsup_{C\to\infty} \limsup_{n\to \infty} |F(a,b,C,n)| = 0. 
\]
In our setting, the parameters $a,b,C$ describe some event $\mathcal{E}_{a,b,C,n}$ depending also on $n$ (where $C$ is introduced at~$T_1$, $a$ at $T_2$ and $b$ at $T_3$).
Thus, saying that $\mathcal{E}_{a,b,C,n}$ holds with probability $1-o^\star(1)$ means that for every $\varepsilon>0$, there is an $a_0$ such that, for every $a\geq a_0$, there is a $b_0=b_0(a)$ such that, for every $b\geq b_0$, there is a $C_0=C_0(b)$ such that, for every $C\geq C_0$, there is an $n_0=n_0(C)$ such that $\mathbb{P}(\mathcal{E}_{a,b,C,n})\geq 1-\varepsilon$ for all $n\geq n_0$. In essence, for any $\varepsilon > 0$ we \emph{can choose} $a,b\in\mathbb{N}$ and $C\in\mathbb{R}^+$ such that for all sufficiently large $n$ the desired event happens with probability at least~$1-\varepsilon$.

\begin{lemma}\label{lem: T_2 consequence}
    There exist a constant $c>0$ such that with probability $1-o^\star(1)$,
\[
    L n\leq X_{T_3} \leq Un,
    \quad \text{where}\quad
    T_3 := T_2 + b,
\]
$T_2$ is as in Lemma~\ref{lem: T_2} and $U,L\in (1/2,1)$ are given by
\begin{equation}
\label{eq:lu}
\begin{aligned}
    L &:=L(a,b,C,n)
        = f^{(b)}\big(1/2+\gamma^{-a-\Phi}(1-c\gamma^{-a})\big) \cdot (1-C^{-1/4}), \\
    U&:=U(a,b,C,n)
        = f^{(b)}\big(1/2+\gamma^{-a-\Phi}\big)\cdot(1+C^{-1/4}).
\end{aligned}
\end{equation} 
\end{lemma}  
Since $f$ is of majority-type, $f^{(b)}(1/2+\varepsilon)\to 1$ as $b\to \infty$ and for any $\varepsilon>0$.
Thus, taking the limits~$C$ and $b$ to infinity in the previous lemma yields $U,L \to 1$ whenever $a$ is large enough.
Therefore, at $T_3$, an appropriate choice of parameters ensures that almost all vertices agree on the majority opinion.
In the next step we translate the conclusion of Lemma~\ref{lem: T_2 consequence}
to the runtime via the notion of \emph{stochastic dominance} --
for random variables $Z_1$ and $Z_2$, we write $``Z_1\succsim Z_2"$, if $\Prob(Z_1\geq x)\geq \Prob(Z_2 \geq x)$ for every $x\in \mathbb{R}$. Similarly, $Z_1\precsim Z_2$ is defined with $``\leq"$ instead of $``\geq"$. 
\begin{lemma}\label{lem: lower and upper bound runtime}
Let $x,x'\in[1/2,1]$ with $x\leq x'$. Then $R_n(x n) \precsim R_n(x'n)$.
\end{lemma}
The (easy) proof is in Section~\ref{subsec: proof of lower and upper bound runtime}. 
As a third milestone, for $1/2<x<1$, the next lemma asserts that starting with $x n$ vertices having the majority opinion, the process reaches consensus in approximately $\log_m\ln n-|\ln (1-x)|$ rounds. The proof is in Section~\ref{subsec: runtime 2 proof}.
\begin{lemma}
\label{lem: lower and upper bound runtime 2} 
    There are $x_0 > 0$ and functions $\varepsilon_+,\varepsilon_-:[x_0,1] \to \mathbb{R}$ such that $\varepsilon_+(x), \varepsilon_-(x) \to 0$ as $x \to 1$ with the following properties. For  $x_0 < x < 1$, uniformly and whp,
    \begin{align*}
        R_n(xn) &\precsim \left\lceil \log_m\ln n-\log_m|\ln (1-x)|+\varepsilon_+(x)+o(1) \right\rceil, \\
        R_n(x n) &\succsim \left\lceil \log_m\ln n-\log_m|\ln (1-x)|-\varepsilon_-(x) - o(1) \right\rceil.
    \end{align*}
\end{lemma}
Based on the two previous lemmas the main statement in Theorem~\ref{thm: main result} follows by proving that $R_n(Un)$ and $R_n(L n)$ coincide with probability $1-o^\star(1)$.
This is established in Section~\ref{subsec: Proof main result}. 
Note that if all these events occur, then the majority opinion remains unchanged after $t_1$ and is therefore also the dominating opinion.

\subsection{Technical preliminaries}\label{subsec: technical preliminaries}

We will need the following useful properties about the functions~$g$ and $h$ in Theorem~\ref{thm: main result} and~Corollary~\ref{cor: main 1}. The proofs are in Sections~\ref{sec: proof of prop of g} and~\ref{sec: proof of prop of h}.
\begin{lemma}\label{lem: prop of g}
    The function $g:\mathbb{R}\to \mathbb{R}$, given by
    \[
    g(x)=1-x+\lim_{a\to \infty, a\in\mathbb{N}} \lim_{b \to \infty, b\in \mathbb{N}} b-a- \log_m|\ln( f^{(b)}(1/2-\gamma^{-a-x}))|,
    \]
    is well-defined, continuous, and $1$-periodic.
\end{lemma}

\begin{lemma}\label{lem: prop of h}
    The function $h:\mathbb{R}\to \mathbb{R}$, given by
    $
    h(x) = g(x)+x,
    $
    is continuous and injective.
\end{lemma}
Lemma~\ref{lem: prop of g} implies that $g$ attains its minimum and maximum so that there are $x,y\in\mathbb{R}$ with $g(y)=\inf g$, $g(x)=\sup g$, and $x\in(y,y+1)$. Moreover, by Lemma~\ref{lem: prop of h} and $h(y+s)=h(y)+s$ for~$s\in\mathbb{Z}$,~$h$ is strictly increasing on $[y,y+1]$. Thus, $\sup g-\inf g<1$, since $g(x)-g(y)+x-y=h(x)-h(y)<1$. 
All the missing pieces and proofs of all lemmas are presented in the next section.

\section{Proofs}

We first prove Theorems \ref{thm: X wins} and \ref{thm: main result} using the steps presented in the proof overview. Afterwards we establish Corollaries \ref{cor: main 2}, \ref{cor: main 1}, \ref{cor: concentration for large d}, and Lemmas \ref{lem: tails Z}, \ref{lem: CLT} -- \ref{lem: lower and upper bound runtime 2} in the following sections.

\subsection{Proof of Theorem \ref{thm: main result}}\label{subsec: Proof main result}

By Lemma~\ref{lem: T_2 consequence}, with probability $1-o^\star(1)$ we have $Ln\leq X_{T_3}\leq Un$, where $L,U$ are as in~\eqref{eq:lu} and,
\[
    t_1 = \lfloor \log_\gamma C \rfloor,
    \quad
    T_2 = t_1 + \lfloor \log_\gamma(n/\tilde{X}_{t_1}) \rfloor - a \ge t_1
    \quad
    \text{and}
    \quad
    T_3 = T_2 + b,
\]
where $a,b\in\mathbb{N}$ and $C \in\mathbb{R}_{>1}$. Since it will facilitate our computations, we assume from now on that $\log_\gamma C\in\mathbb{N}$, so that in particular $T_3= \log_\gamma C+\lfloor \log_\gamma(n/\tilde X_{t_1})\rfloor-a+b$. By Lemma~\ref{lem: lower and upper bound runtime}
\[
    R_n(Ln)\precsim R_n(X_{T_3}) \precsim R_n(Un)
\]
with probability $1-o^\star(1)$. Thus, using Lemma~\ref{lem: lower and upper bound runtime 2}, there are functions $\varepsilon_+,\varepsilon_-:[1/2,1] \to \mathbb{R}$ such that $\varepsilon_+(x), \varepsilon_-(x) \to 0$ as $x \to 1$ and such that with probability $1-o^\star(1)$,
\begin{equation}
\begin{aligned}
    R_n(X_{T_3}) &\precsim\left\lceil \log_m\ln n -\log_m|\ln (1-U)|+\varepsilon_+(U)+o(1) \right\rceil, \\
    R_n(X_{T_3}) &\succsim \left\lceil \log_m\ln n-\log_m|\ln (1-L)|-\varepsilon_-(L) - o(1) \right\rceil.
\end{aligned}
\label{inequ: bound R U,L}
\end{equation}
For this to be useful, we need to compare the logarithmic terms involving $U$ and $L$. The next lemma does exactly that. The proof is in Section~\ref{subsec: Other proofs}.
\begin{lemma}\label{lem: l and u close}
Let $U,L$ be as in~\eqref{eq:lu}. Then, with probability $1-o^\star(1)$, 
\begin{equation*}
     \Big|\log_m|\ln(1-U)|-\log_m|\ln(1-L)|\Bigr| =o^\star(1) \quad \text{and} \quad U,L= 1-o^\star(1).
\end{equation*}
\end{lemma}
Since $f$ is of majority-type, it is symmetric around $1/2$ and increasing; indeed $f$ is convex on $[0,1/2]$ and concave on $[1/2,1]$.
Thus, $L \leq 1-f^{(b)}(1/2-\gamma^{-a-\Phi})\leq  U$, where $\Phi =  \{\log_\gamma(n/\tilde X_{t_1})\}$ as in Lemma~\ref{lem: T_2}.
Together with~\eqref{inequ: bound R U,L} this implies, with probability $1-o^\star(1)$,
\begin{align*}
    R_n(X_{T_3}) &\precsim\left\lceil \log_m\ln n-\log_m|\ln f^{(b)}(1/2-\gamma^{-a-\Phi})|+o^\star(1) \right\rceil, \\
    R_n(X_{T_3}) &\succsim \left\lceil \log_m\ln n-\log_m|\ln f^{(b)}(1/2-\gamma^{-a-\Phi})|-o^\star(1) \right\rceil.    
\end{align*}
Note that $R_n(X_0)\stackrel{d}{=}R_n(X_{T_3})+T_3$, where $T_3= \log_\gamma C+\lfloor \log_\gamma(n/\tilde X_{t_1})\rfloor-a+b$ on an event with probability $1-o^\star(1)$. Thus, recalling that
\[
    h(x)=1+\lim_{a\to \infty, a\in\mathbb{N}} \lim_{b \to \infty, b\in \mathbb{N}} b-a- \log_m|\ln( f^{(b)}(1/2-\gamma^{-a-x}))|,
\]
we obtain with probability $1-o^\star(1)$,
\begin{align*}
    R_n(X_{0}) &\precsim\left\lceil \log_m\ln n+h(\Phi)+\lfloor\log_\gamma(n/\tilde X_{t_1})\rfloor+\log_\gamma C-1+o^\star(1) \right\rceil, \\
    R_n(X_{0}) &\succsim \left\lceil \log_m\ln n+h(\Phi)+\lfloor\log_\gamma(n/\tilde X_{t_1})\rfloor+\log_\gamma C-1-o^\star(1) \right\rceil.    
\end{align*}
Note that $h(x+s)=h(x)+s$ for any $s\in\mathbb{Z}$, since $h(x)=g(x)-x$ and $g$ is $1$-periodic by Lemma~\ref{lem: prop of g}.
Therefore, 
\[
h(\Phi)+\lfloor\log_\gamma(n/\tilde X_{t_1})\rfloor+\log_\gamma C = h\left(\log_\gamma(Cn/\tilde X_{t_1})\right).
\]
Recall that $\tilde X_{t_1} = M_{t_1}+n/2$, where $M_{t_1}$ is the number of vertices with majority opinion at $t_1$. By~\eqref{eq:cltM},
\[
\sup_{x\in\mathbb{R}}|\Prob(M_{t_1}\geq x) - \Prob(n/2+|Z_C|C\sqrt{n}\geq x)| = o(1),
\quad\text{where}\quad Z_{C} \stackrel{d}{=} \mathcal{N}(\mu_{t_1}/C,\sigma_{t_1}^2/C^2).
\]
From this, we readily deduce that 
\[
\sup_{x\in\mathbb{R}}|\Prob(Cn/\tilde X_{t_1}\geq x) - \Prob(\sqrt{n}/|Z_C|\geq x)| = o(1).
\]
Moreover, recall that $Z_C\stackrel{d}{\to}Z$ as $C\to\infty$. Thus, since $h$ is continuous by Lemma~\ref{lem: prop of h}, with probability $1-o^\star(1)$, uniformly for any $z\in \mathbb{Z}$,
\begin{equation}
\begin{aligned}
    \Prob(R_n(X_0) \geq z ) &\leq \Prob(\left\lceil \log_m\ln n+h(\log_\gamma(\sqrt{n}/|Z|))-1+o^\star(1) \right\rceil \geq z) + o^\star(1), \\
    \Prob(R_n(X_0) \geq z ) &\geq \Prob(\left\lceil \log_m\ln n+h(\log_\gamma(\sqrt{n}/|Z|))-1-o^\star(1) \right\rceil \geq z) - o^\star(1).    
\end{aligned}
\label{final bound on R}
\end{equation}
Furthermore, $h(\log_\gamma(\sqrt{n}/|Z|)$ is a continuous random variable, since $h$ is injective by Lemma~\ref{lem: prop of h}. Hence, the upper and lower bound in~\eqref{final bound on R} converge, uniformly in $z\in\mathbb{Z}$, to
\[
    \Prob\left(
        \left\lceil 
            \log_m\ln n + h\big(\log_\gamma(\sqrt{n}/|Z|)\big)-1
        \right\rceil \geq z\right) .
\]
Finally, the fact that $\Prob(\lceil W-1\rceil\geq z)=\Prob(W\geq z)$ for any continuous random variable $W$ and $z\in \mathbb{Z}$, together with $h(x)=g(x)+x$, prove Theorem \ref{thm: main result}.

\subsection{Proof of Theorem \ref{thm: X wins}}

Recall from the end of Section~\ref{sec: proof overview} that, with probability $1-o^\star(1)$, eventually all vertices agree on the majority opinion at $t_1 = \log_\gamma C$, assuming $\log_\gamma C \in \mathbb{N}$. Opinion $X$ (resp.~$Y$) is the majority opinion at $t_1$ exactly when $X_{t_1}-n/2 \geq 0$ (resp.~$X_{t_1}-n/2\leq 0$). Thus,
\[
\Prob(D_n(X_0)=X)=\Prob(X_{t_1}-n/2 \geq 0) + o^\star(1)\quad \text{and}\quad \Prob(D_n(X_0)=Y)=\Prob(X_{t_1}-n/2\leq 0)+o^\star(1).
\]
Moreover Lemma \ref{lem: CLT} shows for $Z_{t_1}\stackrel{d}{=} \mathcal{N}\left(\gamma^{t_1}d, \frac{\gamma^{2t_1}-1}{4(\gamma^2-1)} \right)$ that
\[
\Prob(D_n(X_0)=X) = \Prob(Z_{t_1}\geq 0)+o^\star(1)
\quad\text{and}\quad
\Prob(D_n(X_0)=Y)=\Prob(Z_{t_1}\leq 0)+o^\star(1).
\]
This proves the claim, since $Z_{t_1}/C \stackrel{d}{\to} Z$ with $Z\stackrel{d}{=}\mathcal{N}\left(d,\frac{1}{4(\gamma^2-1)}\right)$.

\subsection{Proof of Corollary \ref{cor: main 2}}

Recall from Lemma~\ref{lem: prop of g} that $g$ is $1$-periodic and continuous, so that in particular $g(Z_n)$ is bounded. The claim then follows by the simple observation that $Z_n = \log_m\ln n + O_{\mathbb{P}}(1)$ and Theorem~\ref{thm: main result}.

\subsection{Proof of Corollary \ref{cor: main 1}}
Let $x,y\in[0,1)$ and $(n_i)_{i\in\mathbb{N}}$ be a strictly increasing sequence in $\mathbb{N}$ such that
\begin{equation}\label{eq: subsequence}
\left\{\tfrac{1}{2}\log_\gamma n_i\right\}\to x
~~
\text{and}
~~
\left\{\log_m\ln n_i\right\} \to y
\quad
\text{as}
\quad
i \to\infty.
\end{equation}
Let $s\in \mathbb{Z}$ and $t=s+\left\lfloor\frac{1}{2}\log_\gamma n_i\right\rfloor+\left\lfloor\log_m\ln n_i\right\rfloor$. Then
\[
\Prob\Bigl(R_{n_i}(X_0) - \left(\left\lfloor\tfrac{1}{2} \log_\gamma n_i\right\rfloor + \lfloor \log_m\ln n_i\rfloor\right)\geq s\Bigr) = \Prob\left(R_{n_i}(X_0)\geq t\right).
\]
Since $g$ is $1$-periodic by Lemma~\ref{lem: prop of g} and $h(x)=g(x)+x$, it follows from Theorem \ref{thm: main result} that 
\[
    \Prob\left(R_{n_i}(X_0)\geq t\right)
    = \Prob
        \Bigl(\left\{\log_m\ln n_i\right\}
            +h\left(\left\{\tfrac{1}{2}\log_\gamma n_i\right\}-\log_\gamma|Z|\right)\geq s
        \Bigr)+o(1).
\]
Using that $Z \stackrel{d}{=} \mathcal{N}\big(d, 1/4(\gamma^2 - 1)\big)$ is continuous, that $h$ is injective by Lemma~\ref{lem: prop of h}, and \eqref{eq: subsequence}, we obtain, as $i \to \infty$,
\[
\Prob\Bigl(\left\{\log_m\ln n_i\right\}+h\left(\left\{\tfrac{1}{2}\log_\gamma n_i\right\}-\log_\gamma|Z|\right)\geq s\Bigr) \to \Prob\Bigl( h(x-\log_\gamma|Z|)+y\geq s\Bigr).
\]
This proves the claim, since $\Prob( h(x-\log_\gamma|Z|)+y\geq s) = \Prob( \left\lfloor h(x-\log_\gamma|Z|)+y\right \rfloor\geq s)$ for any $s\in\mathbb{Z}$.

\subsection{Proof of Lemma \ref{lem: tails Z}}

We recall the following (folklore) tail bounds for the normal distribution. Let~$W\stackrel{d}{=} \mathcal{N}(\mu,\sigma^2)$. Then,
\begin{equation}\label{lem: gaussian tail}
\Prob(|W|\leq z) = \Theta(z), \ \text{as } z\to 0,
\quad\text{and}\quad
\Prob(|W|\geq z) =\exp\left(-\Theta\left(z^2\right)\right), \ \text{as } z\to \infty.
\end{equation}
Let $x,y\in [0,1)$ and $H\stackrel{d}{=}\lfloor h(x-\log_\gamma|Z|) + y\rfloor$, where $Z \stackrel{d}{=} \mathcal{N}\big(d,1/4(\gamma^2-1)\big)$. Since the function $g$ is periodic, it attains its maximum value $g_{\text{max}}$ and minimum value $g_{\text{min}}$ on $\mathbb{R}$. Thus, $z+g_{\text{min}}\leq z+g(z)=h(z) \leq z+g_{\text{max}}$. Therefore, 
\[
\Prob(H\geq t)=\Prob(\lfloor h(x-\log_\gamma|Z|) + y\rfloor\geq t) = \Prob\left(|Z|\leq \Theta\left(\gamma^{-t}\right)\right).
\]
Applying \eqref{lem: gaussian tail} proves the upper tail. Similarly, the lower tail follows by
\begin{align*}
    \Prob(H\leq -t) = \Prob\left(\lfloor h(x-\log_\gamma|Z|) + y\rfloor\leq -t\right) = \Prob\left(|Z|\geq \Theta\left(\gamma^{t}\right)\right) = \exp\left(-\Theta\left(\gamma^{2t}\right)\right).
\end{align*}

\subsection{Proof of Corollary \ref{cor: concentration for large d}}
Recall that $X_0=n/2 +d\sqrt{n}$ with $d=\omega(1)$ and $d\leq (1/2-\varepsilon)\sqrt{n}$ for some $0<\varepsilon< 1/2$. In Section \ref{sec: proof overview}, we use the parameter $C$ to describe a bias of $C\sqrt{n}$ towards the majority opinion at time $t_1=\lfloor\log_\gamma C \rfloor$ and establish that, whp and as $C\to \infty$, the majority opinion does not change after~$t_1$.
Here, we already start with an initial bias of $d\sqrt{n}$ with $d=\omega(1)$ towards $X$, so that, whp, $X$ is the dominating opinion.
Note that we can apply Lemma \ref{lem: T_2} with
\[
    t_1=0,
    \quad T_2 = t_2=\left\lfloor\tfrac{1}{2}\log_\gamma(n/d^2)\right\rfloor -a,
    \quad \Phi=\left\{\tfrac{1}{2}\log_\gamma(n/d^2)\right\},
\]
which are now deterministic quantities, as that lemma addresses only what happens after $t_1$ and the process is Markovian.
We obtain that for sufficiently large $a\in \mathbb{N}$ and $n\in\mathbb{N}$, 
\[
    \frac12 + \gamma^{-a-\Phi} (1-c\gamma^{-a}) \leq f^{(t_2)}(X_{0}/n) \leq \frac12 + \gamma^{-a-\Phi}.
\]
From here on we follow exactly the same steps as in the proof of Theorem~\ref{thm: main result}. In particular, there exists $c>0$ such that with probability $1-o^\star(1)$, for $u,\ell$ given by
\begin{align*}
    \ell =f^{(b)}\big(1/2+\gamma^{-a-\Phi}(1-c\gamma^{-a})\big) \cdot (1-C^{-1/4})
    \quad\text{and}\quad
    u= f^{(b)}\big(1/2+\gamma^{-a-\Phi}\big)\cdot(1+C^{-1/4})
\end{align*}
we have that $un\leq X_{t_3}\leq \ell n, \text{ where } t_3=t_2+b$.
Then, Lemmas~\ref{lem: lower and upper bound runtime} and~\ref{lem: lower and upper bound runtime 2} imply, with probability $1-o^\star(1)$, 
\begin{align*}
    R_n(X_{t_3}) &\precsim\left\lceil \log_m\ln n -\log_m|\ln(1-u)|+\varepsilon_+(u)+o(1)  \right\rceil, \\
    R_n(X_{t_3}) &\succsim\left\lceil \log_m\ln n -\log_m|\ln(1-\ell)|-\varepsilon_+(\ell)-o(1)  \right\rceil,      
\end{align*}
where $\varepsilon_+,\varepsilon_-:[1/2,1] \to \mathbb{R}$ such that $\varepsilon_+(x), \varepsilon_-(x) \to 0$ as $x \to 1$.
As in Lemma~\ref{lem: l and u close}, it follows that $u,\ell$ converge to one and that the double logarithmic terms containing $u,\ell$ are arbitrarily close. 
Recall that the function
\[
    g(x)=1-x+\lim_{a\to \infty, a\in\mathbb{N}} \lim_{b \to \infty, b\in \mathbb{N}} b-a- \log_m|\ln( f^{(b)}(1/2-\gamma^{-a-x}))|
\]
is continuous and $1$-periodic.
Thus, with probability $1-o^\star(1)$,
\begin{align*}
    R_n(X_0) &\precsim \left\lceil \tfrac{1}{2}\log_\gamma (n/d^2)+\log_m\ln n+g\left(\tfrac{1}{2}\log_\gamma(n/d^2)\right)-1+o^\star(1)\right\rceil
\end{align*}
and $\succsim$ also holds if we replace $+o^\star(1)$ with $-o^\star(1)$. This proves the claim.

\subsection{Proof of Lemma \ref{lem: CLT}}\label{subsec: Proof of CLT}
We proceed by induction over $t \in \mathbb{N}_0$. The base case holds by definition of $X_0$. Since $X_0$ is non-random, in the case $t=1$, we can use directly the Lindeberg-Feller CLT so that
\[
    \frac{X_1-n/2}{\sqrt{n}} \stackrel{d}{=} \frac{\Bin\left(n,1/2+\gamma d n^{-1/2}+o(n^{-1/2})\right)-n/2}{\sqrt{n}} \stackrel{d}{\to} \mathcal{N}\left(\gamma d,1/4\right).
\]
For the induction step, assume that the claim holds for $\overline{X}_t$, $t\geq1$, where $\overline{X}_t := n^{-1/2}(X_t-n/2)$. Let $\varepsilon > 0$ and define the event
\[
    E_j := \left\{ j\varepsilon < \overline{X}_t \le (j+1)\varepsilon \right\}\quad \text{for}\ j \in \mathbb{Z}.
\]
Set $A := \varepsilon^{-3/2}$. Then, by the law of total probability,
\begin{equation}
\label{eq:1.8firstest}
    \Prob\left(\overline{X}_{t+1} \leq x \right) = 
    \sum_{|j|\leq A} \Prob\left(\overline{X}_{t+1} \leq x \mid E_j \right) \Prob\left(E_j \right) + O(R_A),
\end{equation}
where $R_A$ is an error term which can be bounded as follows. Note that the continuous mapping theorem asserts that $|\overline{X}_t| \stackrel{d}{\to} |\mathcal{N}(\gamma_t,\sigma_t^2)|$. We obtain
\begin{align*}
    R_A 
    \leq \Prob\left(\left|\overline{X}_t\right| > A \varepsilon\right)  =
    \Prob\left(\left|\mathcal{N}\left(\gamma_t, \sigma_t^2\right)\right| > A \varepsilon\right) + o(1).
\end{align*}
This shows that, if we choose $\varepsilon > 0$ small enough, then we can make $R_A$, up to the $o(1)$ term, as small as we like. It remains to consider the main term in~\eqref{eq:1.8firstest}, in particular $|j|\leq A$. Recall that conditional on $X_t$
\[
    X_{t+1} \stackrel{d}{=} \Bin\left(n,f\left(\frac{1}{2}+\frac{X_t-n/2}{n}\right)\right)  \stackrel{d}{=} \Bin\left(n,f\left(\frac{1}{2}+\frac{\overline{X}_t}{\sqrt{n}}\right)\right).
\]
Recall also that $f(1/2 +x) = 1/2 +\gamma  x+O(x^2)$, since $f$ is of majority-type. Thus, if we condition on $E_j$, then,
\begin{align*}
    X_{t+1}\succeq \Bin\left(n,f\left(1/2+j\varepsilon/\sqrt{n}\right)\right) \stackrel{d}{=} \Bin\left(n,1/2 + \gamma j \varepsilon /\sqrt{n}+ O\left(j^2\varepsilon^2/n\right)\right)  
\end{align*}
and similarly
\[
    X_{t+1} \preceq \Bin\left(n,f\left(1/2+(j+1)\varepsilon/\sqrt{n}\right)\right) \stackrel{d}{=}  \Bin\left(n, 1/2+\gamma j\varepsilon/\sqrt{n} + O\left(\varepsilon/\sqrt{n}+j^2\varepsilon^2/n\right) \right).
\]
These bounds are uniform in $\varepsilon$. Hence, conditioned on $E_j$,
\[  
    X_{t+1} \stackrel{d}{=} \Bin\left(n, 1/2 +\gamma    j\varepsilon/\sqrt{n} + O\left(\varepsilon/\sqrt{n}+j^2\varepsilon^2/n\right) \right).
\]
By applying the Lindeberg-Feller CLT, conditioned on $E_j$,
\[
    \frac{X_{t+1}-n/2-\gamma j\varepsilon \sqrt{n}}{\sqrt{n}/2} + R_\varepsilon \stackrel{d}{\to} \mathcal{N}(0,1),
    \quad\text{for } |j|\leq A,
\] 
where $R_\varepsilon$ denotes an error term tending to zero as $\varepsilon\to0$ and uniformly in $j$.
    Thus,
    \[
    \Prob\left(\overline{X}_{t+1}\leq x \mid E_j \right) = \Prob\left(\gamma j \varepsilon+\mathcal{N}\left(0,1/4\right)-R_\varepsilon \leq x \right) +o(1), \quad \text{for }|j|\leq A.
    \]
    Let $f_t$ denote the density function of $\mathcal{N}(\mu_t,\sigma_t^2)$. By applying the induction hypothesis,
    \begin{align*}
    \Prob\left(E_j \right) &= \Prob\left(j\varepsilon <\mathcal{N}\left(\mu_t, \sigma_t^2\right)\leq (j+1)\varepsilon\right) +o(1) = 
    \int_{(j\varepsilon,(j+1)\varepsilon]} f_t(x)\,dx + o(1),
    \quad \text{for }|j|\leq A.
    \end{align*}
    By the mean value theorem, there exists $\xi\in (j\varepsilon,(j+1)\varepsilon]$ such that $\varepsilon f_t(\xi)=\int_{(j\varepsilon,(j+1)\varepsilon]} f_t(x)\,dx$.
    Since $|j|\leq A=\varepsilon^{-3/2}$, we have uniformly in $|j|\leq A$,
    \begin{align*}
    \left|\log\frac{f_t(\xi)}{f_t(j\varepsilon)}\right|
    \leq \frac{|\xi-j\varepsilon|\,|j\varepsilon-\mu_t|}{\sigma_t^2}+\frac{|\xi-j\varepsilon|^2}{2\sigma_t^2}\leq R'_\varepsilon,
    \end{align*}
    for some error term $R'_\varepsilon$ that tends to zero with $\varepsilon \to 0$.
    Therefore $\Prob\left(E_j \right) = \varepsilon f_t(j\varepsilon)+o(1)+R'_\varepsilon$. So, 
    \begin{align*}
    &\sum_{|j|\leq A} \Prob\left(\overline{X}_{t+1}\leq x \mid E_j \right) \Prob\left(E_j \right) = \sum_{|j|\leq A} \Prob\left(\gamma j \varepsilon +\mathcal{N}(0,1/4)-R_\varepsilon \leq x \right)\,
    \varepsilon\, f_t(j \varepsilon) + o(1)+R'_\varepsilon.
    \end{align*}
    Observe that, as $n\to \infty$ and $\varepsilon\to0$, this becomes the Riemann sum of the integral
    \[
    \int_{\mathbb{R}} \Prob(\mathcal{N}(0,1/4) \leq x-\gamma y)\, f_t(y) \, dy.
    \]
    But this is just the convolution of two Gaussian random variables, i.e, it equals
    \begin{align*}
    \Prob(\mathcal{N}(0,1/4) + \gamma\mathcal{N}(\mu_t,\sigma_t^2) \leq  x) 
    = \Prob(\mathcal{N}(\gamma\mu_t,1/4 + \gamma^2\sigma_t^2) \leq  x).
    \end{align*}
    In conclusion, this shows $
    \Prob\left(\overline{X}_{t+1}\leq x \right) = 
    \Prob(\mathcal{N}(\gamma\mu_t,1/4 + \gamma^2\sigma_t^2) \leq  x) + o(1)$.
    Hence, $\mu_{t+1} = \gamma \mu_t$ and $\sigma^2_{t+1} = 1/4 + \gamma^2\sigma_t^2$ from which the claimed expressions for $\mu_t$ and $\sigma_t$ readily follow. 

\subsection{Proof of Lemma \ref{lem: MiddlePhases}}\label{subsec: Proof of Middle Phases}
Let $t\in \mathbb{N}_0$ and recall $X_{t+1}\stackrel{d}{=}\Bin(n,f(X_t/n))$. Thus
\[
\E[X_{t+1} \mid X_t]=\E[\Bin(n,f(X_t/n)) \mid X_t]  = f (X_t/n)n.
\]
Hence, by Hoeffding's inequality, we get for any $t\in \mathbb{N}_0$ and $x>0$,
\begin{align}\label{inequ:middlephase1}
    \Prob \left(| X_{t+1} - f(X_t/n)n| \geq x \right) \leq 2\exp\left({-2x^2}/n\right).
\end{align}
Furthermore, as preparation, set
\[
    x_t := f^{(t)}(X_{t_1}/n), 
    \quad
    r_0:= \sqrt{C}, 
    \quad \text{and}\quad  
    r_{t+1} := \gamma r_t+\sqrt{C} t \quad \text{for }t\geq 1.
\]
Define the (bad) events
\begin{align*}
    A_t :=\left\{ |X_{t_1+t} - x_t n  | \geq r_t \sqrt{n} \right\},
    \quad t\in\mathbb{N}_0.
\end{align*}
As we will see, $r_t$ grows sufficiently fast so that $A_t$ is (very) unlikely. We claim that
\begin{equation}\label{eq1: proof lem:middlephase}
    \Prob(A_{t} \mid \overline{A}_{t-1}) \leq 2\exp(-2Ct^2),
    \quad t\ge 1.
\end{equation}
Before proving~\eqref{eq1: proof lem:middlephase}, we demonstrate how it  implies the statement of the lemma. To that end, observe that applying the union bound and using that $\Prob(A_0) = 0$ yields
\begin{align*}
    \Prob\left(\bigcap_{t\geq 0} \left \{ | X_{t_1+t} - x_t n | \leq r_t\sqrt{n} \right \} \right) 
    &\geq 1- \Prob(A_0) - \sum_{t \geq 1} \Prob\left(A_{t}\cap \overline{A}_{t-1} \right) \geq 1- \sum_{t \geq 1} \Prob\left(A_{t}\mid \overline{A}_{t-1} \right)\\
    &\geq 1- 2\sum_{t\geq 1} \exp(-2Ct^2)
\end{align*}
that is $\ge 1 - e^{-C}$.
Moreover, by definition
\begin{align*}
r_t &= \sqrt{C}\gamma^t  + \sqrt{C}\left(\sum_{1\leq j\leq t} j\gamma^{-j} \right)\gamma^{t} \leq \sqrt{C}\left(1 + \frac{\gamma}{(\gamma-1)^2}\right)\gamma^t ,
\end{align*}
and the proof is completed, assuming~\eqref{eq1: proof lem:middlephase}, which we establish now. 
The properties of $f$ yield the following convenient upper bound on $f'$. Since $f$ is convex on $[0,1/2]$ and satisfies the symmetry relation $f(1-x)=1-f(x)$, it is concave on $[1/2,1]$. Consequently, $f'$ attains its maximum at $1/2$ and from the local expansion $f(1/2 +x) = 1/2 +\gamma  x+O(x^2)$ we know that $f'(1/2)=\gamma$. 
Moreover, let $z_1, z_2 \in [0,1]$ with $z_1\pm z_2 \in [0,1]$.
By the mean value theorem, there exists $\kappa \in [0,1]$ such that
\begin{align}\label{inequ: lem1.5 1}
| f(z_1\pm z_2) -f(z_1) | \leq | f'(\kappa) | z_2 \leq \gamma z_2.
\end{align}
We choose $z_1=x_{t-1}$ and $z_2=r_{t-1}/\sqrt{n}$. Then, $\overline{A}_{t-1}$ is the same as $|X_{t_1+t-1}/n-z_1|<z_2$.
Thus, since $f$ is monotonically increasing and $x_t=f(z_1)$, applying (\ref{inequ: lem1.5 1}) shows 
\begin{align}\label{inequ:middlephase3}
   \overline{A}_{t-1} \implies | f(X_{t_1+t-1}/n)-x_{t}| \leq \gamma r_{t-1}/ \sqrt{n}.
\end{align}
Applying the triangle inequality yields
\begin{align*}
    \Prob(A_{t}\mid \overline{A}_{t-1}) \leq \Prob\left(| X_{t_1+t}- f(X_{t_1+t-1}/n)n| \geq r_t \sqrt{n} - | f(X_{t_1+t-1}/n) - x_t | n\mid \overline{A}_{t-1}\right).
\end{align*}
Moreover, we use \eqref{inequ:middlephase3} and the definition of $r$ to obtain
\begin{align*}
    \Prob(A_t \mid \overline{A}_{t-1}) &\leq \Prob(| X_{t_1+t}- f(X_{t_1+t-1}/n)n| \geq r_t \sqrt{n} - \gamma r_{t-1} \sqrt{n}) \\
    &= \Prob(| X_{t_1+t}- f(X_{t_1+t-1}/n)n| \geq \sqrt{Cn}t).
\end{align*}
By using~\eqref{inequ:middlephase1} this directly establishes~\eqref{eq1: proof lem:middlephase}.

\subsection{Proof of Lemma \ref{lem: T_2}}\label{subsec: Proof of T_2}

We first show that, whp, $0<\tilde X_{t_1} \leq n^{3/4}$, which guarantees that $t_1 \leq T_2 <\infty$, since $T_2  = t_1 + \lfloor\log_\gamma(n/\tilde X_{t_1})\rfloor-a$.
By Lemma \ref{lem: CLT}, we have
\[
    \Prob(0<\tilde X_{t_1}\leq n^{3/4}) = \Prob(0<|Z_C|C \leq n^{1/4}) + o(1).
\]
Since $CZ_C \stackrel{d}{=} \mathcal{N}(\mu_{t_1},\sigma_{t_1}^2)$
is a continuous random variable and independent of $n$, this proves the claim. 
We next prove the bounds on $f^{(T_2-t_1)}(X_{t_1}/n)$ in~\eqref{eq:t2-2}. 
As a preparation, recall that $X_{t_1}/n \geq 1/2$ and $f(1/2)=1/2$. Since $f(1)=1$ and $f$ is concave on $[1/2,1]$, we get $f^{(t)}(X_{t_1}/n) \in [1/2,1]$ for $t\in \mathbb{N}$. Thus, in the remainder, when we refer to $f$, we actually consider the restriction $f|_{[1/2,1]}$ of $f$ on $[1/2,1]$, which is concave. With this at hand we proceed to establish the upper bound on $f^{(T_2-t_1)}(X_{t_1}/n)$.
Since concave functions stay under their tangents we can assume that $f(1/2+x)\leq 1/2+ \gamma x$. Iterating this directly shows for any $t\in \mathbb{N}$ that
\[
f^{(t)}(1/2+x) \leq 1/2 + \gamma^{t} x.
\]
Substituting $t=T_2-t_1$ and $ X_{t_1}=n/2+\tilde X_{t_1}$ yields
\[
 f^{(T_2-t_1)}(X_{t_1}/n) \leq 1/2 +\gamma^{-a-\Phi},
 \quad
 \text{where}\;
 \Phi = \{\log_\gamma(n/\tilde X_{t_1})\}.
\]
Let us continue with the lower bound on $f^{(T_2-t_1)}(X_{t_1}/n)$.
Since $f$ is twice continuously differentiable, by Taylor's theorem, for any $x\in[0,1/2]$, there is $\overline{x} \in [0,1/2]$ such that
\[
f(1/2+x) = f(1/2) + f'(1/2)x + \frac{f''(1/2+\overline{x})}{2}x^2 =1/2+ \gamma x + \frac{f''(1/2+\overline{x})}{2}x^2.
\]
Moreover, $w_2:=-\min_{[1/2,1]}f''(x)/2>0$ exists and hence for any $x\in[0,1/2]$,
\begin{align}\label{ineq: proof 1.4 (1)}
f(1/2+x) \geq 1/2 + \gamma x (1-\gamma^{-1}w_2x).
\end{align}
We claim that this implies that
\begin{align}\label{ineq: proof 1.4 (2)}
    f^{(t)}(1/2+x) &\geq 1/2+ \gamma^{t} x  \prod_{0\leq j\leq t-1}\left(1- \gamma^{j-1}w_2 x\right),
    \quad \text{for } 0\leq x \leq \gamma^{1-t}/w_2.
\end{align}
The base cases $t\in \{0,1\}$ in~\eqref{ineq: proof 1.4 (2)} follow by definition and~\eqref{ineq: proof 1.4 (1)}. Moreover, the induction step follows by the monotonicity of $f$ and (\ref{ineq: proof 1.4 (1)}),
\begin{align*}
f^{(t+1)}\left(\frac12+x\right) &\geq f\left(\frac{1}{2}+\gamma^{t} x \prod_{j=0}^{t-1}\left(1-\gamma^{j-1}w_2 x \right)\right)
\geq \frac{1}{2} +\gamma^{t+1} x  \left(1- \gamma^{t-1}w_2x\right) \prod_{j=0}^{t-1}\left(1- \gamma^{j-1}w_2 x\right).
\end{align*}
By the Weierstrass product inequality,~\eqref{ineq: proof 1.4 (2)} implies for any $0\leq x \leq \gamma^{1-t}/w_2$
that
\begin{align}
\label{ineq: proof 1.4 (3)}
    f^{(t)}\left(\frac{1}{2}+x\right) &\geq \frac{1}{2}+ \gamma^{t} x  \left(1- \sum_{0\leq j\leq t-1}\gamma^{j-1}w_2 x\right) \geq \frac{1}{2}+ \gamma^{t}x   \left(1- \gamma^{t}w_2  (\gamma-1)^{-1}x\right).
\end{align}
We define $c:=w_2/(\gamma-1)$ and set $t= T_2-t_1= \lfloor \log_\gamma (n/\tilde X_{t_1})\rfloor-a$ and $x=\tilde X_{t_1}/n$, so that 
\[
f^{(T_2-t_1)}\left(\frac{X_{t_1}}{n}\right) =f^{(T_2-t_1)}\left(\frac12+\frac{\tilde X_{t_1}}{n}\right)  \geq \frac{1}{2}+ \gamma^{T_2-t_1} \frac{\tilde X_{t_1}}{n}  \left(1- \gamma^{T_2-t_1}c\frac{\tilde X_{t_1}}{n}\right).
\]
Substituting $\Phi = \{\log_\gamma(n/\tilde X_{t_1})\}$ yields
$f^{(T_2-t_1)}(X_{t_1} / n) \geq \frac{1}{2}+ \gamma^{-a-\Phi} \left(1- c\gamma^{-a}\right)$, as claimed.

\subsection{Proof of Lemma \ref{lem: T_2 consequence}}\label{subsec: proof of T_2 consequences}

We will combine the results of Lemmas~\ref{lem: MiddlePhases} and~\ref{lem: T_2}. First, choose $a_0$ as in Lemma~\ref{lem: T_2} so that there exists a constant $c>0$ such that for all $a\geq a_0$, whp,
\[
    \frac{1}{2}< \frac12 + \gamma^{-a-\Phi} (1-c\gamma^{-a}) \leq f^{(T_2-t_1)}(X_{t_1}/n) \leq \frac12 + \gamma^{-a-\Phi}< 1.
\]
Since, $f$ is monotonically increasing, this implies for any $b\in\mathbb{N}$ and $T_3=T_2+b$, whp,
\begin{equation}
\label{eq1: proof T_2 cons}
\frac{1}{2}<f^{(b)}\left(\frac12 + \gamma^{-a-\Phi} (1-c\gamma^{-a})\right) \leq f^{(T_3-t_1)}(X_{t_1}/n) \leq f^{(b)}\left(\frac12 + \gamma^{-a-\Phi}\right)<1.
\end{equation}
Then, Lemma~\ref{lem: MiddlePhases} together with~\eqref{eq1: proof T_2 cons} imply for a (possibly different) constant $c>0$, for $C > 1$ and with probability at least $1-o(1)-e^{-C}$,
\[
    \tilde L n\leq X_{T_3} \leq \tilde Un,
\]
where
\begin{equation*}
\begin{aligned}
    \tilde L &= f^{(b)}\big(1/2+\gamma^{-a-\Phi}(1-c\gamma^{-a})\big) \cdot (1-c\gamma^{-a+b}\sqrt{Cn}/\tilde X_{t_1}), \\
    \tilde U &= f^{(b)}\big(1/2+\gamma^{-a-\Phi}\big)\cdot(1+c\gamma^{-a+b}\sqrt{Cn}/\tilde X_{t_1}).
\end{aligned}
\end{equation*}
It remains to get a handy presentation of the error terms. Note that 
\[
\Prob\left(c\gamma^{-a+b}\sqrt{Cn}/\tilde X_{t_1} \leq C^{-1/4}\right) = \Prob\left( \tilde X_{t_1}/(\sqrt{n}C)\geq C^{-1/4} c\gamma^{-a+b}\right). 
\]
Applying the convergence result in~\eqref{eq:cltM} shows
\[
\Prob\left(c\gamma^{-a+b}\sqrt{Cn}/\tilde X_{t_1} \geq C^{-1/4}\right) = \Prob\left(|Z_C|\geq C^{-1/4} c\gamma^{-a+b}\right) + o(1).
\]
We assume from now on that $\log_\gamma C \in \mathbb{N}$ and recall that, as $C\to \infty$, $|Z_C|\stackrel{d}{\to} Z$ with $Z\stackrel{d}{=} \mathcal{N}(d,1/4(\gamma^2-1))$.
Thus, with room to spare, we choose $C_0$ such that for any $C\geq C_0$, 
\[
\Prob\left(c\gamma^{-a+b}\sqrt{Cn}/\tilde X_{t_1} \leq C^{-1/4}\right)\geq 1-C^{-1/4}/2.
\]
Consequently, with probability at least $1-o(1)-e^{-C}-C^{-1/4}/2$ and with $L,U$ as in~\eqref{eq:lu} we obtain that $L n\leq X_{T_3} \leq Un$, which completes the proof.

\subsection{Proof of Lemma \ref{lem: lower and upper bound runtime}}
\label{subsec: proof of lower and upper bound runtime}

Let $f$ be a majority type update function and $(X_t)_{t \in\mathbb{N}_0}$ and $(X'_t)_{t \in\mathbb{N}_0}$ be two independent protocols with associated update function $f$. Assume that $1/2 \leq x\leq x'\leq 1$ and
\[
 X_0 = xn 
 \quad\text{and}\quad
 X'_0= x'n.
\]
Let $M_t = \max\{X_t,n-X_t\}$ and $M_t' = \max\{X'_t,n-X'_t\}$ be the number of vertices that have adopted the majority opinion at time $t\in\mathbb{N}_0$.
We show that $M_t \precsim M_t'$ by induction over $t \in \mathbb{N}_0$, which proves the claim. The base cases $t\in\{0,1\}$ hold by assumption and basic properties of the binomial distribution.
For the induction step, let $t\geq 2$ and $M_t \precsim M_t'$. Note that since $Y_t = n-X_t$ we can also write $M_t = \max\{Y_t, n-Y_t\}$, and so,
\[
    M_{t+1} \stackrel{d}{=} \max\left\{W_{t+1},n-W_{t+1}\right\},
    \quad \text{where}\quad
    W_{t+1} \stackrel{d}{=}\text{Bin}(n,f(M_t/n-1/2)).
\]
An analogous statement holds for $M_{t+1}'$. 
Writing $W_{t+1}^s\stackrel{d}{=}\text{Bin}(n,f(s/n-1/2))$, by the law of total probability, 
\begin{align*}
\Prob(M_{t+1} = k) &=\sum_{n/2 \le s \le n}\Prob(M_{t+1} = k \mid M_t=s)\Prob(M_t=s) \\&=\sum_{n/2 \le s \le n}\Prob(\max\left\{ W_{t+1}^s,n- W_{t+1}^s\right\} = k )\Prob(M_t=s).
\end{align*}
The induction hypothesis asserts for any $n/2 \le s \le n$ that $\Prob(M_t=s)\leq \Prob(M_t'=s)$ and therefore, doing the same steps backwards,
\begin{align*}
\Prob(M_{t+1} = k) &\leq \sum_{n/2 \le s \le n}\Prob(\max\left\{ W_{t+1}^s,n- W_{t+1}^s\right\} = k )\Prob(M_t'=s) \\
&= \sum_{n/2 \le s \le n}\Prob(M_{t+1}' = k \mid M_t'=s)\Prob(M_t'=s) =\Prob(M_{t+1}' = k).
\end{align*}

\subsection{Proof of Lemma \ref{lem: lower and upper bound runtime 2}}\label{subsec: runtime 2 proof}
In order to gain sufficient control of $R_n(xn)$ we prove the following auxiliary lemma, which, similar to Lemma~\ref{lem: MiddlePhases}, provides control for approximating the process by iterates of $f$. 
\begin{lemma}
\label{lem: ending phase concentration}
Let $t^\star\in\mathbb{N}$ and $X_{t^\star}=n-Y_{t^\star} = x n$ for $1/2<x<1$.  Then, whp
\[
    \bigcap_{0 \leq t \leq 2 \log_m\ln n} \left\{\left|Y_{t^\star+t} - f^{(t)}(Y_{t^\star}/n)n \right|\leq \max\left\{Y_{t^\star+t}n^{-1/20}, n^{1/4}\right\} \right\}.
\]
\end{lemma}
The proof is at the end of the section. Note that we shift our focus from~$X$, the majority opinion, to~$Y$, the minority opinion. This change is natural, since we will utilize the majority-like behavior of~$f$ near zero, namely~$f(x)=\beta x^m+O(x^{m+1})$ as $x\to 0$. Moreover, note that the error bound in the lemma is much better than the one in Lemma~\ref{lem: MiddlePhases}, however, at the cost of being applicable for short periods of time only. 
Finally, let us remark that we will show that, starting from~$Y_{t^\star}$, the process needs less than~$2 \log_m\ln n$ rounds to terminate and consequently Lemma \ref{lem: ending phase concentration} is handy.

Next, we identify $n^{1-1/m}$ as a critical threshold in the following sense. If the number of vertices holding opinion $Y$ is significantly larger than $n^{1-1/m}$, then the process does not terminate in the next round whp. In contrast, if there are substantially less than $n^{1-1/m}$ vertices with opinion $Y$ left, all vertices will agree (on opinion $X$) after the next round whp.
\begin{lemma}\label{lem: threshold sqrt(n)}
    Let $t, t' \in \mathbb{N}$ such that $Y_t \geq n^{1-1/m}\ln n$ and $Y_{t'} \leq n^{1-1/m}/\ln n$. Then
    \[
    \Prob(Y_{t+1} = 0 \mid Y_t) = o(1)
    \quad 
    \text{and}
    \quad
    \Prob(Y_{t'+1} > 0\mid Y_{t'}) = o(1).
    \]
\end{lemma}
The proof is also at the end of this section. 
With those properties at hand we are now able to prove Lemma \ref{lem: lower and upper bound runtime 2}.
To that end, define
\begin{align*}
    \overline{t}_u := \min \left\{t \in \mathbb{N}_0 : Y_{t^\star +t}\leq n^{1-1/m}/\ln n\right\} 
    \quad
    \text{and}
    \quad
    \overline{t}_\ell := \min \left\{t \in \mathbb{N}_0 : Y_{t^\star+t}\leq n^{1-1/m}\ln n\right\}.
\end{align*}
Thus, by Lemma \ref{lem: threshold sqrt(n)}, whp,
\[
\overline{t}_\ell +1\precsim R_n(xn) \precsim \overline{t}_u +1.
\]
Let $y_{t^\star}=Y_{t^\star}/n = 1-x_{t^\star} = 1-x$. Lemma \ref{lem: ending phase concentration} shows that $(1+o(1))Y_{t^\star+t} = f^{(t)}(y_{t^\star})n + O(n^{1/4})$
for $t\leq 2 \log_m\ln n$.
Thus, for $n$ large enough 
\begin{align*}
    f^{(t)}(y_{t^\star})n \leq n^{1-1/m}/(2\ln n) &\implies Y_{t^\star+t} \leq n^{1-1/m}/\ln n
\end{align*}
as well as
\[
    Y_{t^\star+t} \leq n^{1-1/m} \ln n  \implies f^{(t)}(y_{t^\star})n \leq 2n^{1-1/m} \ln n.
\]
Setting
\[
    t_u := \min \left\{t \in \mathbb{N}_0 : f^{(t)}(y_{t^\star})n\leq \frac{n^{1-1/m}}{2\ln n } \right\},
    \quad
    t_\ell := \min \left\{t \in \mathbb{N}_0 : f^{(t)}(y_{t^\star})n\leq 2n^{1-1/m}\ln n \right\},
\]
we hence obtain whp
    $t_\ell +1\precsim R_n(xn) \precsim t_u +1$.
In the remainder of the proof we pin down $t_u$, $t_\ell$ by determining the smallest $t\in\mathbb{N}_0$ such that
\begin{equation}\label{inequ: End Phases1}
f^{(t)}(y_{t^\star}) \leq c_n n^{-1/m}, \quad \text{for } c_n \in \left\{1/(2\ln n ),2\ln n  \right \}. 
\end{equation}
To achieve this, note that since $f$ is of majority-type, we have $f(y) = \beta y^m +O(y^{m+1})$ as $y \to 0$. So, there are $0< y_0\leq 1/2$ and $r \geq 0$ such that $0<\beta y^m -ry^{m+1} \leq f(y) \leq \beta y^m + r y^{m+1}<1$ for $0< y \leq y_0$. Then, by induction, 
\begin{align*}
\label{inequ: End Phases2} 
    0
    <
    y^{m^t}(\beta -ry)^{(m^t-1)/(m-1)} 
    \leq
    f^{(t)}(y) \leq y^{m^t}(\beta +ry)^{(m^t-1)/(m-1)}
    <
    1,
    \quad \text{for } 0 <y \leq y_0.
\end{align*}
It is straightforward to verify that there are functions $\varepsilon_\star, \widetilde{\varepsilon_\star}:[0,1]\to\mathbb{R}$ such that $\varepsilon_\star(y), \widetilde{\varepsilon_\star}(\delta)\to 0$ as $y,\delta\to 0$ and such that
\[
    \tau = \log_m|\ln\delta|-\log_m|\ln y| + \varepsilon_\star(y)+ \widetilde{\varepsilon_\star}(\delta)
\]
solves the equation $y^{m^\tau}(\beta + O(y))^{(m^\tau-1)/(m-1)} = \delta$. Taking $\delta = c_n n^{-1/m}$ as in~\eqref{inequ: End Phases1} and noting that 
\[
    \log_m|\ln c_n n^{-1/m}| = \log_m \ln n - 1 + o(1)
\]
then yields the statement in the lemma for all $x = x_{t^\star}$ such that $y_{t^\star} = 1 - x \le y_0$.

\medskip
\noindent
To conclude this section we prove the following auxiliary lemma and then Lemmas~\ref{lem: ending phase concentration} and~\ref{lem: threshold sqrt(n)}.
\begin{lemma}\label{lem: aux lem 2.5}
Let $0 < \delta < 1/4$. Then, the event $\bigcap_{0 \le t \le 2\log_m\ln n}C_t$ occurs whp, where
\[
    C_t := \left \{\left|Y_{t+1} - \mathbb{E}[Y_{t+1}\mid Y_t] \right| \leq  \mathbb{E}[Y_{t+1}\mid Y_t]^{1/2+\delta} + n^\delta \right \}
    \quad
    \text{and}
    \quad
    \mathbb{E}[Y_{t+1}\mid Y_t] = f(Y_t/n)n.
\]
\end{lemma}
\begin{proof} 
Let $n' := n^{(m-1)/m+\delta(m-1)/m^2}$. We distinguish two cases according to the value of $Y_{t}$,~i.e.,
\begin{align}
\label{equ: split C_t}
    \Prob(\overline{C_t}) = \Prob(\overline{C_t}\mathbf{1}_{\{Y_t < n'\}})+ \Prob(\overline{C_t}\mathbf{1}_{\{Y_t \geq n'\}}).
\end{align}
Since $f$ is of majority-type, there exist $c_\ell, c_u > 0$ such that 
\begin{align}
\label{inequ: monom bounds f}
    c_\ell x^m \leq f(x) \leq c_ux^m .
\end{align}
    We first consider $\Prob(\overline{C_t}\mathbf{1}_{\{Y_t < n'\}})$. Recall that $Y_{t+1}\stackrel{d}{=}\Bin(n,f(Y_t/n))$. On the event $Y_t < n'$, we have
    \begin{align}\label{inequ: 3}
    \mathbb{E}[Y_{t+1}\mid Y_t] = n f(Y_t/n) \leq c_u Y_t^m n^{1-m} \leq c_u n^{\delta (m-1)/m}.
    \end{align}
    If we assume $\overline{C_t}$, then either
    \begin{align}\label{inequ: 4}
    Y_{t+1} - \mathbb{E}[Y_{t+1} \mid Y_t] > \mathbb{E}[Y_{t+1}\mid Y_t]^{1/2+\delta} + n^\delta
    \end{align}
    or 
    \begin{align}\label{inequ: 5}
    \mathbb{E}[Y_{t+1} \mid Y_t]-Y_{t+1}>\mathbb{E}[Y_{t+1}\mid Y_t]^{1/2+\delta} + n^\delta.
    \end{align}
    Note that $\mathbb{E}[Y_{t+1} \mid Y_t]\geq0$.
    By \eqref{inequ: 3}, there exists $n_0$ such that for all $n\geq n_0$ case \eqref{inequ: 5} is impossible. In the following, we assume that $n\geq n_0$. Thus, by~\eqref{inequ: 4} and~\eqref{inequ: 3},
    $
    Y_{t+1} \geq n^{\delta} \geq c_u^{-1} n^{\delta/m}  \mathbb{E}[Y_{t+1}\mid Y_t].
    $
    By applying Markov's inequality, we obtain
    \begin{align}\label{inequ: first summand}
    \Prob(\overline{C_t}\mathbf{1}_{\{Y_t < n'\}}) \leq \Prob(Y_{t+1} \geq c_u^{-1} n^{\delta/m}  \mathbb{E}[Y_{t+1}\mid Y_t]) \leq c_un^{-\delta/m}.
    \end{align}
    It remains to treat the case $Y_t \geq n'$. Applying~\eqref{inequ: monom bounds f} and the definition of $n'$ imply
    \[
    \mathbb{E}[Y_{t+1}\mid Y_t] = nf(Y_t/n) \geq c_\ell Y_t^m n^{1-m} \geq c_\ell n^{\delta(m-1)/m}.
    \]
    Since $Y_{t+1}\stackrel{d}{=}\Bin(n,f(Y_t/n))$, its conditional variance satisfies $\mathbb{V}[Y_{t+1}\mid Y_t]=nf(Y_t/n)(1-f(Y_t/n))$.
    Furthermore, as $f(x)\in[0,1]$ for any $x\in[0,1]$,
    \[
    \mathbb{E}[Y_{t+1}\mid Y_t] = nf(Y_t/n) \geq nf(Y_t/n)(1-f(Y_t/n)) = \mathbb{V}[Y_{t+1}\mid Y_t].
    \]
    Hence, $\overline{C_t}=\left\{\left|Y_{t+1} - \mathbb{E}[Y_{t+1}\mid Y_t] \right| >  \mathbb{E}[Y_{t+1}\mid Y_t]^{1/2+\delta} + n^\delta \right\}$ implies with room to spare
    \[
    \left| Y_{t+1} - \mathbb{E}[Y_{t+1}\mid Y_t] \right|> \mathbb{E}[Y_{t+1}\mid Y_t]^{1/2+\delta} \geq c_\ell^{\delta}n^{\delta^2 (m-1)/m} \mathbb{V}[Y_{t+1}\mid Y_t]^{1/2}.
    \]
    Applying Chebyshev's inequality yields 
    \begin{align}\label{inequ: second summand}
    \Prob(\overline{C_t}\mathbf{1}_{\{Y_t \geq n'\}}) \leq \Prob(\left| Y_{t+1} - \mathbb{E}[Y_{t+1}\mid Y_t] \right|\geq c_\ell^{\delta}n^{\delta^2 (m-1)/m} \mathbb{V}[Y_{t+1}\mid Y_t]^{1/2}) \leq  c_\ell^{-2\delta}n^{-2\delta^2(m-1)/m}.
    \end{align}
    Combining~\eqref{equ: split C_t},~\eqref{inequ: first summand} and~\eqref{inequ: second summand}, we obtain $\Prob(\overline{C_t}) = n^{-O(1)}$, and the statement in the lemma follows with room to spare from a union bound.
\end{proof}
\begin{proof}[Proof of Lemma \ref{lem: ending phase concentration}]
Since $(X_t)_{t\in \mathbb{N}_0}$ is markovian we can assume without loss of generality that $t^\star = 0$. In particular, $X_0 = xn$ and $Y_0 = (1-x)n$.
We abbreviate
\[
    \Delta_{t} = \left|Y_{t} - f^{(t)}(Y_0/n)n \right|,
    \quad\text{for } 0 \le t \le 2 \log_m\ln n.
\]
For later reference, since $f$ is of majority-type, we have $f(x)=\beta x^m + O(x^{m+1})$ as $x\to 0$. Moreover, the first $m-1$ derivatives of $f$ in zero are zero, since the forward difference representation shows
\begin{align*}
    \frac{d^s}{dx^s} f (0) &= \lim_{h\searrow 0} \frac{\sum_{j=0}^{s} (-1)^{s-j}\binom{s}{j}f(jh)}{h^s} = \lim_{h\searrow 0} \sum_{j=1}^{s} (-1)^{s-j} j^s \binom{s}{j} \frac{f(jh)}{(jh)^s} = 0 
    \quad\text{for }
    1\leq s <m.
\end{align*}
It follows by Taylor's theorem that
\begin{align}\label{equ: formula derivative}
    f'(x)=m\beta x^{m-1} + O(x^{m}).
\end{align}
Hence, there exist $c_u,c_\ell>0$ such that for $x\in[0,1]$,
\begin{align}\label{inequ: f upper lower x^m}
    c_\ell x^m \leq f(x) \leq c_ux^m
    \quad\text{and}\quad
    f'(x)\leq c_ux^{m-1}.
\end{align}
Moreover, if $Y_{t'} < n^{1/4}$ for some $t' \in \mathbb{N}$, then $Y_{t'+1} =0$ whp by Lemma \ref{lem: threshold sqrt(n)}. Hence, let $\overline{T} \in \mathbb{N}_0 \cup \{\infty\}$ be such that, whp,
\begin{align}\label{consequ of Lem 2.8}
    Y_{t} \geq n^{1/4} ~ \text{for all} ~  t\leq \overline{T}
    \quad \text{and} \quad
    Y_{t} < n^{1/4} ~ \text{for all} ~  t > \overline{T}.
\end{align}
In order to prove the lemma, we condition
on the event \eqref{consequ of Lem 2.8} and additionally
on the event in Lemma~\ref{lem: aux lem 2.5}, where, with foresight, $\delta=1/14$. We claim for $n$ large enough and $\lambda:=2^{m+1}c_u/c_\ell$ that
\begin{equation}
\label{eq:boundD_t}
    \Delta_t \leq \lambda^tY_t/n^{1/14},
    ~~\text{for}~~
    t\leq \min\{\overline{T}, 2\log_m\ln n\}
    \quad\text{and}\quad
    \Delta_{t} \leq n^{1/4}
    ~~\text{for}~~
    \overline{T} < t \le 2\log_m\ln n .
\end{equation}
This implies the statement of the lemma, since $\lambda^t = n^{o(1)}$ for $t \leq 2 \log_m \ln n$.
Thus, it remains to prove~\eqref{eq:boundD_t}. We proceed by induction over $t \in \mathbb{N}_0$. Note that $\Delta_{0}=0$, establishing the base case. For the induction step, we assume that the claim holds for some $t\in \mathbb{N}_0$.
Recall that $\mathbb{E}[Y_{t+1}\mid Y_t] = f(Y_t/n)n$. Since $Y_t$ satisfies Lemma~\ref{lem: aux lem 2.5} with $\delta = 1/14$, applying the triangle inequality implies
\begin{equation}\label{inequ: first bound delta}
\begin{aligned}
    \Delta_{t+1} &\leq \left|Y_{t+1}-f(Y_t/n)n \right| +\left|f(Y_t/n)n -  f^{(t+1)}(Y_0/n)n \right| \\
    &\leq (f(Y_t/n)n)^{4/7} + n^{1/14} + n\left|f(Y_t/n) -  f^{(t+1)}(Y_0/n) \right|.
\end{aligned}
\end{equation}
By definition $Y_t/n = f^{(t)}(Y_0/n) \pm \Delta_{t}/n$. Hence, using the mean value theorem, there exists $\xi$ between $Y_t/n$ and $f^{(t)}(Y_0/n)$, so that $0\leq \xi \leq Y_t/n + \Delta_t/n$, and such that
\begin{align*}
\label{inequ: second summand1}
    n \left|f(Y_t/n) -  f^{(t+1)}(Y_0/n) \right|
    = f'(\xi)\Delta_t
    \stackrel{\eqref{inequ: f upper lower x^m}}{\leq}c_u\, (Y_t/n +\Delta_t/n)^{m-1}  \Delta_t.
\end{align*}
This implies that 
\begin{equation}
\label{inequ:2ndboundD_t}
    \Delta_{t+1}
    \le
    (f(Y_t/n)n)^{4/7} + n^{1/14} + c_u\, (Y_t/n +\Delta_t/n)^{m-1}  \Delta_t.
\end{equation}
In what follows we distinguish whether $t > \overline{T}$ or $t \le \overline{T}$.
Suppose $t > \overline{T}$. 
By~\eqref{consequ of Lem 2.8} we have $Y_t\leq n^{1/4}$ and the induction hypothesis yields $\Delta_t\leq n^{1/4}$, too. 
From~\eqref{inequ:2ndboundD_t} we obtain that
    \begin{align*}
    \Delta_{t+1}
    &\leq c_u^{4/7}n^{(4-3m)/7}+n^{1/14}+c_u\,2^{m-1}n^{(4-3m)/4},
    \end{align*}
    which is for sufficiently large $n$ and independently of $t$ at most $n^{1/4}$, as required.
    It remains to consider the case $t\leq \overline{T}$.
Recall that $\lambda=2^{m+1}c_u/c_\ell$. We proceed from~\eqref{inequ:2ndboundD_t}, where we first note that since $t\leq2\log_m\ln n$, the induction hypothesis implies $\Delta_t\leq Y_t$.
Applying $x^m\leq f(x)/c_\ell$ from~\eqref{inequ: f upper lower x^m} and using again the induction hypothesis, this time for the second occurrence of $\Delta_t$, we obtain
    \begin{equation}\label{inequ: second bound delta}
        \Delta_{t+1} \leq (f(Y_t/n)n)^{4/7} + n^{1/14} + \frac{\lambda^{t+1}}{4}  n^{13/14}f(Y_t/n).
    \end{equation}
Since $Y_t$ satisfies Lemma \ref{lem: aux lem 2.5} with $\delta=1/14$. we obtain
$f(Y_t/n)n \leq Y_{t+1}+\left(f(Y_t/n)n\right)^{4/7}+n^{1/14}$.
Note that this guarantees that $f(Y_t/n)n \le 2(Y_{t+1} + n^{1/14})$ whenever $n$ is sufficiently large; to see this, note that if $f(Y_t/n)n$ is larger than some specific constant, then $(f(Y_t/n)n)^{4/7} \le f(Y_t/n)n/2$. Thus,~\eqref{inequ: second bound delta} asserts that
\begin{equation}
\label{eq:Dt+1aux}
    \Delta_{t+1} \leq \left(2Y_{t+1} + n^{1/14}\right)^{4/7} + n^{1/14} + \frac{\lambda^{t+1}}{2} (Y_{t+1} + n^{1/14})\, n^{-1/14}.
\end{equation}
If $Y_{t+1} \ge n^{1/4}$, then this readily implies 
\[
    \Delta_{t+1}
    \le {\lambda^{t+1}}Y_{t+1}n^{-1/14} \left(\frac12 + o(1)\right)
\]
and~\eqref{eq:boundD_t} is established for large $n$. On the other hand, if $Y_{t+1} < n^{1/4}$, then the right hand side in~\eqref{eq:Dt+1aux} is already $o(n^{1/4})$, and~\eqref{eq:boundD_t} is established for large $n$ as well.
\end{proof}

\begin{proof}[Proof of Lemma \ref{lem: threshold sqrt(n)}]
    By~\eqref{inequ: f upper lower x^m}, there exist two constants $c_\ell, c_u > 0$ such that $c_\ell x^m \leq f(x) \leq c_ux^m$. 
    Let $Y_t \geq n^{1-1/m}\ln n$ and recall $Y_{t+1} \stackrel{d}{=} \Bin(n,f(Y_{t}/n))$. Using that $1-x\leq e^{-x}$, we obtain
    \begin{align*}
        \Prob(Y_{t+1}=0\mid Y_t) = \left(1-f(Y_{t}/n)\right)^n \leq \exp(-f(Y_{t}/n)n) \leq \exp(-c_\ell \ln^m n).
    \end{align*}
    Since $m\geq 2$ this proves the first claim.
    To complete the proof, let $Y_{t'} \leq n^{1-1/m}/ \ln n$, then
    \[
    \mathbb{E}[Y_{t'+1} \mid Y_{t'}] =  f(Y_{t'}/n)n \leq c_u Y_{t'}^m n^{1-m} \leq c_u \ln^{-m}n = o(1).
    \]
    This proves the claim by applying Markov's inequality.
\end{proof}

\subsection{Proof of Lemma \ref{lem: prop of g}}\label{sec: proof of prop of g} 
The proof is inspired by the strategy in \cite{debruijn} for analyzing asymptotic problems on iterated functions.
By the properties of $f$ there exists a function $p(x) = O(x)$ as $x\to 0$ such that
\begin{align}\label{equ: f in proof Lem2.2}
f(x) = x^m (\beta +p(x)) \quad \text{and}\quad f^{(t)}(x)=x^{m^t} \prod_{0\leq s \leq t-1} \left(\beta+ p(f^{(s)}(x)) \right)^{m^{t-1-s}}
\quad\text{for } t\in \mathbb{N}.
\end{align}
To shorten the notation, set
\[
z:= 1/2 -\gamma^{-a-x} \quad \text{and}\quad q(x) := \ln(\beta+p(x)) = \ln f(x) - m\ln x.
\]
In the following, we assume that $a,b\in \mathbb{N}$. Applying \eqref{equ: f in proof Lem2.2} with $t=b$ shows
\begin{align*}\label{equ: fct g in proof Lem2.2}
g(x) 
&= 1-x-\Big(\lim_{a\to \infty} a +\log_m\Big| \ln z + \sum_{s\geq 0}m^{-s-1}q(f^{(s)}(z))\Big|\Big).
\end{align*} 
Define the following two functions on $(0,1/2)$
\[
Z(x):= -m \ln x  \quad 
\text{and} \quad
W(x) := -Z(x)+\sum_{s\geq 0} m^{-s}q(f^{(s)}(x))
\]
such that $g(x) = 2-x-(\lim_{a\to\infty}a+\log_m\left|W(z) \right|)$.
Moreover,  $Z(x) = q(x)+Z(f(x))/m$
and consequently, $W$ satisfies the functional equation
\begin{align}\label{equ: Func equation L}
    W(x)=W(f(x))/m,
    \quad\text{for}\quad 0<x<1/2.
\end{align}
In the following, we will describe \emph{all} solutions of~\eqref{equ: Func equation L}, which will eventually enable us to establish every property of $W$ that we need.
We proceed in four steps:
\begin{enumerate}
\itemsep0em 
    \item\label{step1} Describe one solution $W_0$ of~\eqref{equ: Func equation L}.
    \item\label{step2} Show that any solution of~\eqref{equ: Func equation L} has the form $W_0(x)v(x)$, where $v$ is such that $v(f(x))=v(x)$.
    \item\label{step3} Establish continuity of $W$.
    \item\label{step4} Establish that $g$ is $1$-periodic.
\end{enumerate}

\noindent
\underline{1. Base solution $W_0$:} Recall that $f'(1/2)=\gamma > 1$, since $f$ is of majority-type. We first construct a continuous and positive solution $M_0$ of the functional equation 
\begin{align}\label{equ: Schroeder equ}
M_0(f(x)) = \gamma M_0(x),
\quad\text{for}\quad 
0<x<1/2,
\end{align}
so that a solution of \eqref{equ: Func equation L} is given by
$
W_0(x):=m^{\log_{\gamma}(M_0(x))}.
$
Since $f$ is bijective, its inverse $f^{(-1)}$ and also iterates of it are well-defined.
Moreover, from $f(1/2-x) = 1/2 - \gamma x + O(x^2)$, it follows that $f^{(-1)}(1/2-x)=1/2-\gamma^{-1}x+O(x^2)$ as $x\to0$. In particular,
\begin{align}\label{equ: eta}
\eta(x):=\gamma\frac{1/2-f^{(-1)}(1/2-x)}{x} = 1+ O(x) 
\quad\text{as}\quad
x\to 0.
\end{align}
Note that $f^{(-1)}(x)>x$ for $x\in (0,1/2)$ and $f^{(-1)}$ is concave on $[0,1/2]$, since $f(x)<x$ for $x\in (0,1/2)$ and $f$ is convex on $[0,1/2]$. Thus, for $x\in (0,1/2)$, 
we obtain that $\delta_j(x):=1/2-f^{(-j)}(x)$ tends to zero as $j\to \infty$. Using that $\eta(x) = 1 + O(x)$ this directly implies that
\begin{equation}
\label{eq:deltaj+1deltaj}
    \delta_{j+1}(x)
    = \delta_j(x) \gamma^{-1} \eta\big(1/2 - f^{(-j)}(x)\big)
    = \delta_j(x)\gamma^{-1}\big(1+O(\delta_j(x))\big)
\end{equation}
and consequently $\delta_{j+1}(x)\leq \delta_j(x)\gamma^{-1}(1+\varepsilon)$ for any fixed $\varepsilon>0$ and sufficiently large $j$. Hence
\[
    \prod_{j\geq 0} \eta(\delta_j(x)) = \prod_{j\geq 0} (1+O(\gamma^{-j})) <\infty,
    \quad x \in (0,1/2).
\]
Using this, we can define a function on $(0,1/2)$ by
\begin{align*}
    x \mapsto (1/2-x)\prod_{j\geq 0} \eta(\delta_j(x))
    = (1/2-x)\prod_{j\geq 0} \gamma \frac{1/2 - f^{(-j-1)}(x)}{1/2 - f^{(-j)}(x)}
    , 
\end{align*}
which is well-defined, positive, continuous and obviously satisfies~\eqref{equ: Schroeder equ}.

\medskip
\noindent
\underline{2. All solutions:}
We show that any solution $T$ of \eqref{equ: Func equation L} is obtained from $W_0$ by multiplying with a function~$v$ satisfying $v(f(x))=v(x)$. Fix $x_0 \in (0,1/2)$. For any $x \in (0,1/2)$, there is a unique $u \in \mathbb{Z}$ such that $f^{(u)}(x) \in (f(x_0),x_0]$, since $f$ is increasing and $f^{(t)}(x_0) \to 0$ and $f^{(-t)}(x_0) \to 1/2$ as $t \to \infty$. 
Consequently, $T$ is completely determined by its restriction to $(f(x_0),x_0]$, since 
\begin{align}\label{equ: Iteration of T}
T(x) = m^{-u}T(f^{(u)}(x)),
\quad\text{where}\quad
f^{(u)}(x) \in (f(x_0),x_0].
\end{align}
The crucial step is to express $T$ on $(f(x_0),x_0]$ via a function $w$ such that
\[
    T(x) = W_0(x)w(\log_\gamma(M_0(x))), \quad \text{for}\; x \in(f(x_0),x_0].
\]
Note that $w$ is defined on an interval of length one, since $\log_\gamma(M_0(f(x_0)))-\log_\gamma(M_0(x_0))=1$ by~\eqref{equ: Schroeder equ}.
By combining~\eqref{equ: Iteration of T} with~\eqref{equ: Func equation L}, we obtain for $x \in (0,1/2)$ that
\begin{align*}
T(x) &= m^{-u}T(f^{(u)}(x)) = m^{-u}W_0(f^{(u)}(x))w(\log_\gamma(M_0(f^{(u)}(x)))) = W_0(x) w(u+\log_\gamma(M_0(x))).
\end{align*}
Thus, we can choose $w$ to be a $1$-periodic function on $\mathbb{R}$ via the relation $w(x)=w(1+x)$, such that
\[
    T(x) = W_0(x) w(\log_\gamma(M_0(x))).
\]
Finally, setting $v(x) := w(\log_\gamma(M_0(x)))$, we obtain $v(f(x)) = v(x)$, which proves the claim. 

\medskip
\noindent
\underline{3. Continuity:} 
In Step~\ref{step2} we showed that $W$ is determined by its values on $(f(x_0),x_0]$ for some $x_0 \in (0,1/2)$, via the relation $W(x) = m^{-u}W(f^{(u)}(x))$. Since $f^{(u)}(x)$ is continuous for $u\in \mathbb{Z}$, it suffices to show continuity on $[f(x_0),x_0]$. Set
\[
    W^{\leq t}(x)
    = \sum_{0\leq s\leq t} m^{-s}\ln(\beta +p(f^{(s)}(x)))-Z(x),
    \quad t\in\mathbb{N}.
\]
The function $p$ is continuous, as $p(x)=f(x)/x^m-\beta$. 
Thus, $W^{\leq t}(x)$ is continuous for every $t\in\mathbb{N}$ and by definition, $W^{\leq t}(x) \to W(x)$ as $t\to\infty$.
Since $p(x) \to 0$ as $x \to 0$ and if $x < 1/2$ also $f^{(s)}(x) \to 0$ as $s \to \infty$, we get for large $t\in\mathbb{N}$,
\[
\left|W(x)- W^{\leq t}(x)\right| =\left|\sum_{s> t}m^{-s} \ln(\beta +p(f^{(s)}(x)))\right| \leq |\ln(2\beta)|\sum_{s\geq t}m^{-s} \xrightarrow[t \to \infty]{} 0\quad \text{for}\; x \in [f(x_0),x_0].
\]
This proves that $W(x)$ is the uniform limit of continuous functions on $[f(x_0),x_0]$ and therefore continuous. If we write $W(x) = W_0(x)w(\log_\gamma(M_0(x)))$, we can also deduce that $w$ is continuous, since $W(x)$, $W_0(x)$ and $\log_\gamma(M_0(x))$ are continuous.

\medskip
\noindent
\underline{4. $g$ is $1$-periodic:} 
Recall that $g(x) = 2-x-\lim_{a\to\infty}a+\log_m\left|W(z) \right|$ with $z=1/2 -\gamma^{-a-x}$. By Steps~\ref{step2} and~\ref{step3} there is a continuous $1$-periodic function $w$ such that  
\[
    W(x)=m^{\log_{\gamma}(M_0(x))} w(\log_\gamma(M_0(x)))
\quad\text{for } 0<x<1/2,
\]
where, by Step~\ref{step1},
\[
    M_0(x)=(1/2-x)\prod_{j\geq 0} \eta(\delta_j(x)) 
    \quad\text{and}\quad 
    \eta(x)=\gamma\frac{1/2-f^{(-1)}(1/2-x)}{x},
    \quad
    \delta_j(x)=1/2-f^{(-j)}(x).
\]
Next, we study the function $\log_\gamma(M_0(1/2-x))$ (since we are interested in $W(z)$). 
As we showed after~\eqref{eq:deltaj+1deltaj}, for $x\in(0,1/2)$, fixed $\varepsilon>0$ and sufficiently large~$j$, 
\[
    \delta_{j}(1/2-x)\leq \delta_{j-1}(1/2-x)\gamma^{-1}(1+\varepsilon)\leq \delta_0(1/2-x)\gamma^{-j}(1+\varepsilon)^j = x\gamma^{-j}(1+\varepsilon)^j.
\]
Since $\eta(x)=1+O(x)$ by~\eqref{equ: eta}, and $\log_\gamma(1+x)=O(x)$ as $x\to 0$, for sufficiently small $\varepsilon>0$,
\[
\log_\gamma \prod_{j\geq 0} \eta(\delta_j(1/2-x)) = \sum_{j\geq0}\log_\gamma \eta(O(x\gamma^{-j}(1+\varepsilon)^j)) 
= O(x)
\quad \text{as }x\to 0.
\]
Since $M_0(1/2-x)=x\prod_{j\geq 0} \eta(\delta_j(1/2-x))$, we obtain 
$\log_\gamma(M_0(1/2-x)) = \log_\gamma x + O(x)$ as $x\to 0$.
Therefore, 
\begin{align*}
    g(x) &= 2-x-\Big(\lim_{a\to\infty}a+\log_m\left|m^{-a-x+O(\gamma^{-a-x})}w(-a-x+O(\gamma^{-a-x})) \right|\Big)\\
    &=2-\lim_{a\to\infty}\log_m\left|w(-x+O(\gamma^{-a-x})) \right|.
\end{align*}
Thus, the continuity of $w$ shows $g(x) = 2-\log_m|w(-x)|$, so $g$ is a continuous $1$-periodic function.

\subsection{Proof of Lemma~\ref{lem: prop of h}}\label{sec: proof of prop of h}
    By Lemma \ref{lem: prop of g}, $h$ is continuous and $h(x+1)=h(x)+1$. Therefore, it suffices to prove injectivity on an interval of the form $[y,y+1)$, where $y\in\mathbb{R}$ will be chosen later. 
    Define the functions
    \[
    w_t(x):=m^{-t}\ln (f^{(t)}(x)) \quad \text{and} \quad u_t(x):=f^{(t)}(1/2-\gamma^{-t-x}),
    \quad t\in\mathbb{N}.
    \]
    In the following, we assume that $a,b\in\mathbb{N}$. With this notation, note that
    \[
        h(x)= 1+\lim_{a\to\infty}\lim_{b\to \infty}-\log_{m}|w_{b-a}(u_a(x))|.
    \]
    Note that $w_{b-a}(u_a(x))$ is negative for every $x$. In particular, the absolute value in the definition of $h$ does not affect injectivity considerations. Thus, it suffices to show that $\lim_{a\to\infty}\lim_{b\to \infty}w_{b-a}(u_a(x))$ is injective.
    To this end, we prove that there exists $\ell_y$ such that
    \[
    \liminf_{a\to\infty}\liminf_{b\to \infty}w_{b-a}'(u_a(x)) u_a'(x)\geq \ell_y >0
    \quad \text{for } x\in [y,y+1).
    \]
 We first establish that $\liminf_{a\to \infty} u'_a(x)$ is bounded away from zero on $[y,y+1)$. By the chain rule
    \[
    u_a'(x)= \frac{\ln\gamma}{\gamma^{a+x}} \prod_{0\leq j\leq a-1}f'(f^{(j)}(1/2-\gamma^{-a-x})).
    \]
    Note that $f'$ is increasing on $[0,1/2]$, since $f$ is convex there. Thus, choosing $y$ such that $1/2-\gamma^{-y}>0$ and using $1/2-\gamma^{-a-x}\geq 1/2-\gamma^{-a-y}$, since $x\in[y,y+1)$, we obtain
    \[
    u_a'(x)\geq \frac{\ln\gamma}{\gamma^{a+y+1}}\prod_{0\leq j\leq a-1}f'(f^{(j)}(1/2-\gamma^{-a-y})).
    \]
    Moreover, $f'(1/2)=\gamma$ and $f''(1/2)=0$, since $f''$ is continuous and $f$ is convex on $[0,1/2]$ and concave on $[1/2,1]$. 
    By Taylor's theorem $f'(1/2-x)=\gamma-O(x^2)>0$ for $x< 1/2$.
    Since $f$ is convex and convex functions stay over their tangents, we further obtain $f^{(j)}(1/2-\gamma^{-a-y})\geq 1/2-\gamma^{-a-y+j}>0$ for $0\leq j \leq a-1$. Hence,
    \[
        u_a'(x)
        \ge
        \frac{\ln\gamma}{\gamma^{y+1}}\prod_{1\leq j \leq a}\left(1-O(\gamma^{-2y-2j})\right).
    \]
    By the Weierstrass product inequality this is $\ge 1-\sum_{1\leq j\leq a}O(\gamma^{-2y-2j})$. Since every factor in the product is positive, by summing a geometric series, we obtain that there is a constant $c>0$ such that $\prod_{1\leq j \leq a}\left(1-O(\gamma^{-2j})\right)\geq c$, so that $u_a'(x)$ is bounded away from zero, as required.

We conclude by proving that $\liminf_{a\to\infty}\liminf_{b\to \infty}w_{b-a}'(u_a(x))$ is bounded away from zero on $[y,y+1)$. As preparation, we first establish an upper bound on $u_a(x)$ that is uniform in $x$ and $a$.
Applying~\eqref{ineq: proof 1.4 (3)} with $t=a$ and $x=\gamma^{-a-x}$, and using that $f^{(a)}(1-x) = 1-f^{(a)}(x)$ we obtain
\[
u_a(x) = f^{(a)}(1/2-\gamma^{-a-x}) \leq 
\frac{1}{2}- \gamma^{-x}   \left(1- \gamma^{-x}w_2  (\gamma-1)^{-1}\right),
\quad \text{for } 0\leq \gamma^{-x-1} \leq 1/w_2,
\]
where $w_2>0$ is the constant from~\eqref{ineq: proof 1.4 (3)}. 
We choose $y$ sufficiently large such that every $x\in[y,y+1]$ satisfies $\gamma^{-x}\leq (\gamma-1)/(2w_2)$. Hence,
\[
    u_a(x) \leq \frac{1}{2}- \frac{1}{2\gamma^{y+1}} =:u_y.
\]
With this at hand, let $b':=b-a$ and $z:=u_a(x)\leq u_y$.
The derivative  
\begin{align}\label{equ: deriv of w_b'}
    w'_{b'}(z) = \frac{1}{m^{b'} f^{(b')}(z)}\prod_{0\leq j \leq b'-1}f'(f^{(j)}(z)).
\end{align}
Recall that $f(x)=\beta x^m+O(x^{m+1})$ and $f'(x)=\beta m x^{m-1}+O(x^m)$ as in~\eqref{equ: formula derivative}. Since $f'(x)>0$ for $0<x<1/2$,
\begin{align}\label{equ: quotient}
    \frac{f'(x)^{m/(m-1)}}{m f'(f(x))^{1/(m-1)}} = 1+O(x)>0
    \quad\text{for}\quad 0<x<1/2.
\end{align}
Consequently, for $0<x<1/2$,
\begin{align*}
    f'(x)^{1/(m-1)}\prod_{0\leq j \leq b'-1}f'(f^{(j)}(x))
    =
    (1+O(x))m f'(f(x))^{1/(m-1)}  \prod_{1\leq j \leq b'-1}f'(f^{(j)}(x)).
\end{align*}
By iterating this, we obtain for $0<x<1/2$,
\[
    f'(x)^{1/(m-1)}\prod_{0\leq j \leq b'-1}f'(f^{(j)}(x))
    = m^{b'}f'(f^{(b')}(x))^{1/(m-1)} \prod_{0\leq j\leq b'-1}(1+O(f^{(j)}(x))).
\]
Recall $z=u_a(x) \in (0,1/2)$ so that from \eqref{equ: deriv of w_b'} it follows that
\begin{equation}
\label{eq:w'b'2}
    w'_{b'}(z)
    \geq
    f'(z)^{-1/(m-1)}\frac{f'(f^{(b')}(z))^{1/(m-1)}}{f^{(b')}(z)} \cdot  \prod_{0\leq j\leq b'-1}(1+O(f^{(j)}(z))).
\end{equation}
We conclude by proving that the three terms on the right-hand side are bounded away from zero as $b'\to\infty$ and $a\to\infty$.
Since $f$ is convex on $[0,1/2]$, the derivative $f'$ is increasing, so that $f'(z)^{-1/(m-1)} \geq f'(1/2)^{-1/(m-1)} = \gamma^{-1/(m-1)}$. Thus, $f'(z)^{-1/(m-1)}$ is bounded away from zero. 
Regarding the second term, note that $f'(x)=\beta m x^{m-1}+O(x^m)$ implies 
\[
\frac{f'(x)^{1/(m-1)}}{x}=(\beta m)^{1/(m-1)}+O(x).
\]
Moreover, $f^{(b')}(z)$ converges to $0$ as $b'\to\infty$, since $f(x)<x$ for $x\in(0,1/2)$. Thus,
\[
    \frac{f'(f^{(b')}(z))^{1/(m-1)}}{f^{(b')}(z)}
    = (\beta m)^{1/(m-1)} +O(f^{(b')}(z)) \overset{b'\to\infty}{\longrightarrow}(\beta m)^{1/(m-1)}
\]
that is bounded away from zero and independent of $a$.
We finally consider $\prod_{0\leq j< b'}(1+O(f^{(j)}(x)))$, where every factor is positive by~\eqref{equ: quotient}.
Note that
$f^{(j)}(z)\leq f^{(j)}(u_y)$, since $z\leq u_y$ and $f$ is increasing. Using that $f(x) \le c_u x^m$ from~\eqref{inequ: f upper lower x^m} we infer that $f^{(j)}(u_y)$ converges superexponentially fast to zero as $j\to\infty$. So, the product is bounded away from zero, even as $b'\to\infty$ uniformly in $a$.

\subsection{Proof of Lemma~\ref{lem: l and u close}}\label{subsec: Other proofs}

Recall from Lemma~\ref{lem: T_2 consequence} that on an event with probability $1-o^\star(1)$ we have that $\Phi =  \{\log_\gamma(n/\tilde X_{t_1})\}$,
\begin{align*}
    L = f^{(b)}\big(1/2+\gamma^{-a-\Phi}(1-c\gamma^{-a})\big) \cdot (1-C^{-1/4})
    \quad
    \text{and}    
    \quad
    U= f^{(b)}\big(1/2+\gamma^{-a-\Phi}\big)\cdot(1+C^{-1/4}).
\end{align*}
We first show that $U$ and $L$ converge to one as $C\to \infty$ and $b\to \infty$ whenever $a$ is sufficiently large. 
Since $\Phi\in[0,1)$, for sufficiently large $a$,
\[
\lim_{C\to\infty} L \geq  f^{(b)}\big(1/2+\gamma^{-a-1}(1-c\gamma^{-a})\big).
\]
Moreover, since $f$ is of majority-type, $f^{(b)}(1/2+\delta)\to1$ as $b\to \infty$ and for any $\delta >0$. Therefore, using $L \leq U \leq 1$, we conclude that
$
\lim_{b\to \infty}\lim_{C\to \infty} L = \lim_{b\to \infty}\lim_{C\to \infty} U =1.
$
Thus, we have just established that $U=1-o^\star(1)$ and $L=1-o^\star(1)$ with probability $1-o^\star(1)$.

It remains to show that the logarithmic terms containing $U,L$ are very close. 
Since $f(x)>x$ for $x\in (1/2,1)$ and $f$ is monotonically increasing,
for any $\delta>0$ and sufficiently large $a,b\in\mathbb{N}$, there exists $C_0>0$, such that for any $C\geq C_0$,
\begin{align*}
    U\leq f^{(b)}\big(1/2+\gamma^{-a-\Phi+\delta}\big)
    \quad\text{and}\quad
    L \geq f^{(b)}\big(1/2+\gamma^{-a-\Phi-\delta}\big).
\end{align*}
Hence, for such combinations of $a,b,C$
\begin{align*}
&\Big|\log_m|\ln(1-U)|-\log_m|\ln(1-L)|\Bigr| \\
&\leq \Big|\log_m|\ln (1-f^{(b)}(1/2+\gamma^{-a-\Phi+\delta})|-\log_m|\ln (1-f^{(b)}\big(1/2+\gamma^{-a-\Phi-\delta}\big)| \Big|.
\end{align*}
Lemma~\ref{lem: prop of g} asserts that the function
\[
    g(x)=1-x+\lim_{a\to \infty, a\in\mathbb{N}} \lim_{b \to \infty, b\in \mathbb{N}} b-a- \log_m|\ln( f^{(b)}(1/2-\gamma^{-a-x}))|,
\]
is well-defined and continuous. Let $\varepsilon>0$.  Then there is $a_0\in\mathbb{N}$ such that for any $a\geq a_0$ there is $b_0(a)\in\mathbb{N}$ such that for any $b\geq b_0$ and all sufficiently large $C$
\begin{align}\label{inequ: eps/2}
    &\Big|\log_m|\ln (1-U)|-\log_m|\ln (1-L)| \Big|\leq|g(\Phi+\delta)-g(\Phi-\delta)|+\frac\varepsilon2.
\end{align}
Moreover, by continuity, choosing $\delta$ sufficiently small guarantees that $|g(\Phi+\delta)-g(\Phi-\delta)|\leq \varepsilon/2$. This proves the claim.

\section{Majority-type Protocols}
\label{sec:examples}

Apart from \texttt{$k$-maj} we consider two additional classes of protocols, as mentioned in the introduction:
\begin{itemize}
\itemsep0pt
    \item \texttt{rand-$k$-maj.} Let $K$ be a bounded random variable taking values in $\mathbb{N}_{\geq 3}$. The update rule for \texttt{rand-\allowbreak $k$-\allowbreak maj} with respect to $K$ is as follows. Each vertex samples the number of considered neighbors $k$ according to the distribution $K$ and then applies the \texttt{$k$-maj} update rule.
    \item \texttt{$k$-neighb-rand.} Let $Q$ be a random variable taking values in $\{2,\dots,k-1\}$, symmetric around $k/2$,~i.e., $\Prob(Q=q)=\Prob(Q=k+1-q)$ for all $q\in\{2,\dots,k-1\}$. Let $X$ and $Y$ be the two possible opinions.
    As in $k$-majority, every vertex samples $k$ neighbors uniformly at random, with replacement. Next, every vertex chooses a threshold $q$ according to $Q$ and adopts opinion $X$ if there are at least $q$ sampled neighbors with opinion $X$.
\end{itemize}
In this section we show that our main results apply to all aforementioned protocols by establishing that in all cases the associated update function is of majority type as required in Definition~\ref{def: majority-type}.

\medskip
\noindent
\texttt{$k$-maj:} The probability that a vertex has opinion $X$ at time $t+1$ is given by
\begin{alignat*}{2}
    f_k(X_t/n)&= \sum_{0\leq i <  k/2 }\binom{k}{i} (X_t/n)^{k-i} (1-X_t/n)^{i}, \quad 
    &&\text{if }k \text{ is odd}, \\
    f_k(X_t/n)&=\sum_{0\leq i <  k/2 }\binom{k}{i} (X_t/n)^{k-i} (1-X_t/n)^{i} +\binom{k}{k/2}(X_t/n)^{k/2} &&(1-X_t/n)^{k/2}/2 \\
    &=\sum_{0\leq i <  (k-1)/2 }\binom{k-1}{i} (X_t/n)^{k-1-i} (1-X_t/n)^{i}, \quad 
    &&\text{if }k \text{ is even}.
\end{alignat*}
Consequently, if $k$ is odd, then $f_k$ and $f_{k+1}$ are the same and 
\[
f_k(x) = \binom{k}{\lfloor k/2 \rfloor} x^{\lceil k/2 \rceil} +O(x^{\lceil k/2 \rceil+1}) 
\quad
\text{and} 
\quad
f_k(1/2 +x) =1/2 +\binom{k-1}{\lfloor k/2\rfloor} 2^{1-k}k x + O(x^2).
\]
The property $f_k(x)=1-f_k(1-x)$ follows by symmetry of the binomial distribution or direct computation.
Thus,~Theorem~\ref{thm: main result} applies with
\[
m = \lceil k/2\rceil, \quad \beta =\binom{k}{\lfloor k/2 \rfloor} \quad \text{and} \quad \gamma = \binom{k-1}{\lfloor k/2\rfloor} 2^{1-k}k.
\]
Figure \ref{fig:f_k} illustrates $f_k$ for several $k$ and a numerical approximation of $g$ is given in Figure \ref{fig:g3-maj}. 

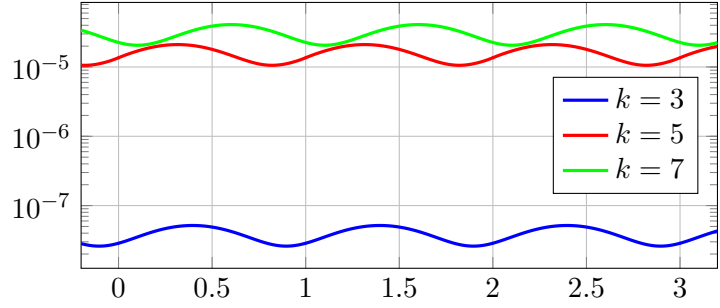
\begin{figure}
    \centering
\begin{tikzpicture}

\begin{axis}[
    ymode=log,
    width=10cm,
    height=5.1cm,
    xmin=-0.2, xmax=3.2,
    xtick distance=0.5,
    ytick={1e-8,1e-7,1e-6,1e-5,1e-4},
    minor tick num=0,
    grid=major,
    legend style={
    at={(0.97,0.5)},
    anchor=east
    }
]

\addplot [color=blue, very thick] table [x=x, y=g3] {multi-c.dat};
\addlegendentry{$k=3$}

\addplot [color=red, very thick] table [x=x, y=g5] {multi-c.dat};
\addlegendentry{$k=5$}

\addplot [color=green, very thick] table [x=x, y=g7] {multi-c.dat};
\addlegendentry{$k=7$}

\end{axis}
\end{tikzpicture}
\caption{The plot depicts a numerical approximation of $g$, appearing in the limiting distribution of \texttt{$k$-maj}.
We display the vertically shifted function $\overline{g}(x)=g(x)-2\inf g +\sup g$ in order to make its continuity and the $1-$periodicity more apparent. The amplitude $\sup g-\inf g$ is approximately $2.6\cdot 10^{-8}$ for $k=3$, $1.1\cdot 10^{-6}$ for $k=5$ and $5.5\cdot 10^{-6}$ for $k=7$.}
\label{fig:g3-maj}
\end{figure}

\medskip
\noindent
\texttt{rand-$k$-maj:} Let $K$ be a fixed bounded random variable taking values in $\mathbb{N}_{\geq 3}$. Each vertex samples its number of considered neighbors according to $K$. Thus, for $k\in \mathbb{N}_{\geq 3}$, the probability that a vertex applies the $k$-majority update rule is given by $\Prob(K=k)$. Accordingly, the function $f$ takes the form
\begin{align}\label{ex: f for random-k-maj}
    f(x) = \sum_{k\geq 3} \Prob(K=k)f_k(x).
\end{align}
Since $K$ is bounded, the function $f$ is well-defined and satisfies $f(x)=1-f(1-x)$, as it holds for every $f_k$.
Let $k_1$ be the smallest integer such that $\Prob(K=k_1)>0$. We assume that $k_1$ is odd, otherwise, the result follows by $f_{k_1}(x) = f_{k_1-1}(x)$. The crucial values $m$, $\beta$ and $\gamma$ are  
\[
m = \lceil k_1/2\rceil, \quad \beta = \Prob(K=k_1) \binom{k_1}{\lfloor k_1/2 \rfloor} \quad \text{and} \quad
\gamma = \mathbb{E}\left[\binom{K-1}{\lfloor K/2\rfloor} 2^{1-K}K\right].
\]

\medskip
\noindent
\texttt{$k$-neighb-rand:} Let $Q$ be a fixed symmetric random variable taking values in $\{2,\dots,k-1\}$. For given $q\in\{2,\dots,k-1\}$, the probability that a vertex adopts opinion $X$ at time $t+1$ is $\Prob(\Bin(k,X_t/n)\geq q)$. Hence, the function $f$ is given by
\[
f(x)= \sum_{2\leq q\leq k-1} \Prob(Q=q) \sum_{q\leq i\leq k} \binom{k}{i}x^i (1-x)^{k-i}.
\]
Since $Q$ is symmetric, it follows that $f(x)=1-f(1-x)$. Moreover, let $q_1$ be the smallest integer such that $\Prob(Q=q_1)>0$, we obtain
\[
m = q_1, \quad \beta = \Prob(Q=q_1) \binom{k}{q_1} \quad \text{and} \quad \gamma = \mathbb{E}\left[\binom{k-1}{Q-1}\right]k 2^{1-k}.
\]
The assumption $Q \ge 2$ then implies that $m \ge 2$, as required.

\bibliographystyle{abbrv}
\bibliography{references}

\end{document}